\documentclass[12pt]{amsart}
\usepackage[top=3cm, bottom=3cm, left=2cm, right=2cm]{geometry}
\usepackage[colorlinks,citecolor = red, linkcolor=blue,hyperindex]{hyperref}
\usepackage{euscript,eufrak,verbatim, mathrsfs}
\usepackage[psamsfonts]{amssymb}
\usepackage{graphicx}
\usepackage{float, tikz}
\usepackage{amsfonts,amsmath,amstext,amsbsy, amsopn,amsthm, amscd,tikz-cd}
\usepackage{extpfeil}
 \usepackage[all, cmtip]{xy}
\usepackage{upref, xcolor, dsfont}
\usepackage{enumerate}
 \usepackage{times}
\newtheorem{theorem}{Theorem}[section]
\newtheorem*{theorem*}{Theorem B}
\newtheorem{lemma}[theorem]{Lemma}
\newtheorem{proposition}[theorem]{Proposition}
\newtheorem{corollary}[theorem]{Corollary}
\theoremstyle{definition}

\newtheorem*{definition*}{Definition}
\theoremstyle{remark}

\newtheorem*{remark*}{Remark}

\newcommand{\R}{\mathbb{R}}
\newcommand{\N}{\mathbb{N}}
\newcommand{\Z}{\mathbb{Z}}

\newcommand{\C}{\mathbb{C}}
\newcommand{\E}{\mathbb{E}}

\newcommand{\PP}{\mathbb{P}}

\newcommand{\Var}{\mathrm{Var}}
\newcommand{\Cov}{\mathrm{Cov}}
\newcommand{\an}{\text{\, and \,}}

\newcommand{\VU}{Y(u)}
\numberwithin{equation}{section}
\begin{document}
\title[Microcanonical cascades and random homeomorphisms]{Microcanonical cascades and random homeomorphisms}
\author
{Xinxin Chen}
\address
{Xinxin CHEN: School of Mathematical Sciences
Beijing Normal (Teachers) University
Beijing 100875, China}
\email{xinxin.chen@bnu.edu.cn}
\author
{Yong Han}
\address
{Yong HAN: School of Mathematical Sciences, Shenzhen University, Shenzhen 518060, Guangdong, China}
\email{hanyong@szu.edu.cn}
\author
{Yanqi Qiu}
\address
{Yanqi Qiu: School of Fundamental Physics and Mathematical Sciences, HIAS, University of Chinese Academy of Sciences, Hangzhou 310024, China}
\email{yanqi.qiu@hotmail.com, yanqiqiu@ucas.ac.cn}
\author{Zipeng Wang}
\address{Zipeng WANG: College of Mathematics and Statistics, Chongqing University, Chongqing
401331, China}
\email{zipengwang2012@gmail.com, zipengwang@cqu.edu.cn}
\date{\today}
\begin{abstract}
 We  give a complete solution to the Mandelbrot-Kahane problem for the microcanonical cascade measures by determing their exact Fourier dimensions.  We also discuss the Frostman regularity as well as the bi-H\"older continuity of the Dubins-Freedman random homeomorphisms. 
\end{abstract}
\subjclass[2020]{Primary  60G57, 42A61, 46B09; Secondary  60J80, 60G46}
\keywords{Mandelbrot microcanonical cascades; Fourier dimension;  Vector-valued martingales; H\"older regulrity of random homeomorphism; Branching random walks}
\maketitle
\setcounter{tocdepth}{2}
\setcounter{equation}{0}
\section{Introduction}\label{sec-intro}

Influenced by the turbulence theory developed in seminal works of Kolmogorov-Obukhov-Yaglom, Mandelbrot introduced  multiplicative cascade  models.  Mandelbrot's theory aims to construct and analyze random fractal measures on the unit interval $[0,1]$, and the original theory had two main  formulations: the microcanonical (or conservative) form and the canonical form \cite[p. 67]{MA-book}.

In the 1970s,  Mandelbrot  formulated several key conjectures and fundamental questions about his multiplicative canonical cascade measures, including the non-degeneracy of the measures, the existence of their finite moments, and the Hausdorff dimension of these measures. Mandelbrot's conjectures were soon validated by Kahane and Peyrière in \cite{Kahane-Peyriere-advance}. Their results were subsequently generalized by Holley-Waymire \cite{HW92}, Ben Nasr \cite{N87}, and Waymire-Williams \cite{WW95}, who extended the analysis to include the multifractal properties of the microcanonical cascade measure as particular cases (see \cite[Corollary 2.1]{GWF99}). Moreover,  microcanonical cascades have many applications in stock prices \cite{Mandelbrot1997}, river flows and rainfalls \cite{GW90}, wavelet analysis \cite{RSGW03}, Internet WAN traffic \cite{FGW98}.  The reader is referred to \cite{DL83, Liu00, Barral2001, Fan-JMPA, BM1, BM2} for more related works.

In 1976, Mandelbrot  \cite{M76} (see also his selected works \cite[p. 402]{MA-book}) also recognized the roles of harmonic analysis on multiplicative cascade models. He anticipated  that the understanding of multiplicative cascades may at long last benefit from results in harmonic analysis. In particular, he raised the question of the optimal Fourier decay of cascade measures. In 1993, Kahane \cite{Ka-93} revisited Mandelbrot's problem and formulated a comprehensive  open program to investigate the Fourier decay of natural random fractal measures.

By introducing the vector-valued martingale theory into the harmonic analysis of   cascade measures, we established in our recent work \cite{CHQW24a} (announced in \cite{CHQW24b}) a complete solution to the Mandelbrot-Kahane problem for the Mandelbrot canonical cascade measure by giving the exact Fourier dimension formula.  The main goal of this paper is to give a complete solution to the Mandelbrot-Kahane problem for the microcanonical cascade measures.

\subsection{Statement of the main result}

Consider the random vector  $W = (W_0, W_1)$ with positive coordinates ($W_0>0$ and $W_1>0\, \, a.s. $) 
such that  (throughout the whole paper, we assume that $W_0\not \equiv 1/2$)
\begin{align}\label{def-W0W1}
	W_0+W_1 = 1\, \, a.s. \an \E[W_0] =  \E[W_1] = 1/2. 
\end{align}
Let $\mu_\infty$ be the Mandelbrot microcanonical cascade measure associated to the random vector $W$ (its precise definition will be briefly recalled in \S \ref{sec-construct} below).  Denote the Fourier transform of $\mu_\infty$ by 
\[
\widehat{\mu}_\infty (\zeta)=\int_{[0,1]} e^{2\pi i t \zeta}d\mu_\infty(t),\quad \zeta\in \mathbb{R}.
\]
The Fourier dimension of $\mu_\infty$ is defined by
\[
\dim_F(\mu_\infty): =\sup\Big\{D \in[0,1]:  |\widehat{\mu}_\infty(\zeta)|^2 = O(|\zeta|^{-D}) \quad \text{as $|\zeta|\to\infty$}\Big\}.
\]
Set
\begin{align}\label{def-DF}
	D_F= D_F(W) : = \log_2 \Big( \frac{1}{\E[W_0^2]+\E[W_1^2]}\Big)  = \log_2 \Big( \frac{1}{2 \E[W_0^2]}\Big).
\end{align}
Observe  that
$
\E[W_0^2] = \E[W_1^2] \in (1/4, 1/2)
$, hence $D_F\in (0,1)$. 

 \begin{theorem}[Fourier dimension]\label{main-thm}
	Almost surely, we have $\mathrm{dim}_F(\mu_\infty)=D_F. $
\end{theorem}

\begin{remark*}
By a classical Fourier analysis result, we know that Theorem~\ref{main-thm} implies that the distribution function of $\mu_\infty$ is $\gamma$-H\"older continuous for all $\gamma\in (0,D_F/2)$.  However,  usually, the optimal H\"older exponent of the distribtution function  cannot be obtained from the Fourier dimension of the measure $\mu_\infty$.   
\end{remark*}

\subsection{Discussions on Dubins-Freedman random homeomorphisms}

Given a  microcanonical cascade measure $\mu_\infty$ on $[0,1]$ associated with a random vector $W=(W_0,W_1)$, its distribution  function gives rise to a random self-homeomorphism of $[0,1]$ (see \S\ref{sec-DF} for more details)
 \[
 F_\infty(t)= \mu_\infty([0,t]), \quad t\in [0,1]. 
 \]
 In 1967,  such random homeomorphisms were constructed by Dubins and Freedman  \cite{DF-65} (see also \cite[p. 305]{WW97}) without using the cascade theory, thus we refer them as Dubins-Freedman random homeomorphisms.  As noted by Graf, Mauldin and Williams \cite{RA-Homeo}, the Dubins-Freedman random homeomorphisms are connected to an old question posed by S. Ulam of defining a natural probability measure on the group of self-homeomorphism of the unit circle.  We note that the H\"older regularity of the  Dubins-Freedman random homeomorphisms is one of the key ingredients in  Kozma and Olevskiĭ's recent advancements  \cite{KOzma, KO22,  KO23}   on a problem of Luzin  and related questions about the  improvement of the  convergence rate of Fourier series of a continuous function  by a random change of variable.

 The H\"older regularity of the distribution function of a measure  can be equivalently formulated as its upper-Frostman regularity (also known its Frostman dimension).   Barral-Jin-Mandelbrot  \cite{BJM10b} studied the Frostman regularity of general Mandelbrot cascade measures (including complex case) on the interval $[0,1]$. One can consult \cite{BJM10a,BJ10} for more related results.  For sub-critical Mandelbrot cascade measures, the optimal exponents of the Frostman regularity are obtained by  Barral, Kupiainen, Nikula, Saksman and Webb  \cite[Theorem 4]{BKNSW}.  Moreover, generalized Frostman regularity are obtained in \cite{BKNSW} for critical cascade measures (note that the critical  cascade measures have zero Fourier dimensions).

Define 
 \begin{align}\label{eqn-opt}
 \gamma_{o}^{+} =  \gamma_{o}^{+}(W) :=\sup_{p>0} \frac{\log_2\big [ (\E[W_0^p+W_1^p])^{-1}\big] }{p} ;
 \end{align}
  \begin{align}\label{eqn-OPT}
 \gamma_{o}^{-}  =   \gamma_{o}^{-} (W): =  \inf_{p>0} \frac{\log_2 \big[ \E[W_0^{-p}+W_1^{-p}]\big] }{p} . 
 \end{align}
 
 \begin{proposition}[Frostman regularity]\label{frostman-measure}
 Almost surely,   there exists  $C>1$ (a random constant), such that for any subinterval $I\subset[0,1]$, 
 \begin{align}\label{ineq-2-frost}
\frac{1}{C}  |I|^{\gamma_o^{-}} \le \mu_\infty(I) \le C  |I|^{\gamma_o^{+} },
 \end{align}
 and $\gamma_{o}^{\pm}$ are both sharp in the sense that,  for any $\delta>0$, 
 \[
 \sup_I \frac{\mu_\infty(I)}{|I|^{\gamma_o^{+} +\delta}} = \infty \an  \inf_{I} \frac{\mu_\infty(I)}{|I|^{\gamma_o^{-} -\delta}} = 0  \quad  a.s. 
 \]
If $\gamma_o^{-} = +\infty$, then the left-hand side of \eqref{ineq-2-frost}  is understood as $|I|^{+\infty} = 0$ for $|I|<1$. 
 \end{proposition}

 \begin{remark*}
  By Lemmas \ref{lem-opt} and \ref{lem-Opt} below, we shall see that $\gamma_{o}^{+} \in (0,1)$ and $\gamma_{o}^{-} \in (1,\infty]$.   One note that, in our setting, by establishing an entropy-type inequality for 2D random vectors (see Proposition~\ref{prop-p-2} below),  we always have 
  \[
  \gamma_o^{+}> D_F/2. 
  \]
 \end{remark*}
 \begin{remark*}
 It is worthwhile to mention that,  in general, the  upper Frostman regularity cannot guarantee the positive Fourier dimension of  a measure. For instance, the classical Cantor-Lebesgue measure $\mu_{\mathrm{CL}}$ on the one-third Cantor set of $[0,1]$ is upper Frostman regular with the exponent $\frac{\log 2}{\log 3}$,  but the Fourier coefficients of $\mu_{\mathrm{CL}}$ has no Fourier decay since $\widehat{\mu_{\mathrm{CL}}} (3n) = \widehat{\mu_{\mathrm{CL}}} (n)$ for any $n\in \N$. That is, $\dim_{F}(\mu_{\mathrm{CL}})= 0$.  
 \end{remark*}
 
  \begin{remark*}
 For Kahane’s Gaussian multiplicative chaos (GMC) on the circle, the Frostman regularity was established by Astala-Jones-Kupiainen-Saksman  \cite[Theorem 3.7]{AJKS}.  For the most recent developments on harmonic analysis of GMC, we refer to \cite{LQT24,LQT25} and the references therein.
 \end{remark*}

 \begin{corollary}\label{holder-function}
 Almost surely, the Dubins-Freedman random homeomorphism $F_\infty$  is  H\"older continuous of order $\gamma_{o}^{+}$ and   the inverse Dubins-Freedman random homeomorphism $F_\infty^{-1}$  is  H\"older continuous of order $(\gamma_o^{-})^{-1}$. Moreover, the H\"older exponents are sharp. 
 \end{corollary}

  \begin{remark*}
   The microcanonical cascade used by Kozma-Olevskiĭ is  related to the special random vector 
\begin{align}\label{orig-W}
	W\stackrel{d}{=}(U, 1-U) \quad   \text{with  $U$ being uniformly distributed on $(0,1)$}.
\end{align}
In this special case,    Kozma and Olevskiĭ \cite[Remarks after Lemma 1.4]{KOzma} already obtained \eqref{ineq-2-frost}.    We believe that the formalism developed by Kozma and Olevskiĭ could be extended beyond the case \eqref{orig-W}. 
  \end{remark*}

 \subsection{Outline of the proof of Theorem \ref{main-thm}}

 One of the key ingredients is the estimate of  the following  Sobolev-type norm on $\mu_\infty$.

 {\flushleft \it  Step 1. Polynomial Fourier decay via vector-valued martingale estimates.}

 \begin{proposition}\label{prop-vec}
 For any $\varepsilon>0$, there exists $q>2$ large enough such that 
 \[
\E\Big[ \Big\{\sum_{n\in \Z} \big(|n|^{\frac{D_F}{2}-\varepsilon}   \cdot |\widehat{\mu}_\infty(n)|\big)^q\Big\}^{2/q} \Big]<\infty. 
\]
 \end{proposition}

 {\flushleft \it Step 2. Optimality of the polynomial exponent: fluctuation of branching random walks. }

 Consider 
 \begin{align*}
\mathcal{M}_n^{(2)}=\frac{1}{2^n}\sum_{|u|=n}\prod_{j=1}^{n} \frac{W_{u_j}(u|_{j-1})^2}{\E[W_0^2]},\quad n\geq 1.
\end{align*}
One can verify that $(\mathcal{M}_n^{(2)}: n\geq 1)$ is a positive martingale and hence converges to a limit denoted by $\mathcal{M}_\infty^{(2)}$. In Lemma~\ref{lemma-minfinity-positive} below, we shall prove that $\PP(\mathcal{M}_\infty^{(2)}>0) =1$.

 Denote by 
 \[
 \varrho=\E[|\widehat{\mu}_\infty(1)|^2] \quad \an\quad \varpi=\E[\widehat{\mu}_\infty(1)^2].
 \]
 In Lemma~\ref{lem-0-av} below, we shall show that $\varpi$ is a real number. Indeed, we show that $\varpi <0$ and  $\varrho\pm \varpi>0$. 
\begin{proposition}\label{prop-convergence-in-law}
Along the dyadic subsequence, the fluctuation of the  rescaled  Fourier coefficients $\widehat{\mu}_{\infty}(2^n)$ is given by 
\[
2^{\frac{n D_F}{2}} \widehat{\mu}_{\infty}(2^n)\xrightarrow[n\to\infty]{d} \sqrt{\mathcal{M}_\infty^{(2)}} \cdot \mathcal{N}_\C(0,\Sigma),
\]
where   $\mathcal{N}_\C(0,\Sigma)$ is the complex random Gaussian with covariance matrix given by 
\[
\Sigma=\frac{1}{2}\begin{pmatrix}
\varrho+\varpi&0\\
0&\varrho-\varpi\\
\end{pmatrix}.
\]
 Moreover $\mathcal{N}_\C(0,\Sigma)$ and  $\mathcal{M}_\infty^{(2)}$ are independent.
\end{proposition}

 {\flushleft \it Step 3. The almost sure equality $\dim_F(\mu_\infty) = D_F$. }
 
 Proposition~\ref{prop-vec} implies the almost sure  inequality  
 $
 \dim_F(\mu_\infty) \ge D_F
 $
and  Proposition~\ref{prop-convergence-in-law} implies the  almost sure converse inequality 
 $
 \dim_F(\mu_\infty) \le D_F. 
 $
See \S \ref{sec-main-thm} for the details. 

\begin{remark*}
Let $D_2(\mu_\infty)$ be the so-called correlation dimension defined as 
\[
D_2(\mu_\infty):=\liminf_{n\rightarrow\infty}\frac{\log\sum_{I\in\mathcal{D}_n} \mu_\infty(I)^2}{-n\log 2},
\]
where $\mathcal{D}_n$ denotes the set of dyadic subintervals   in $[0,1]$ of  length $1/2^n$. The  almost sure upper bound $\dim_F(\mu_\infty) \le D_F$ can also be obtained by using the standard inequality  $\dim_F(\mu_\infty) \le D_2(\mu_\infty)$ from potential theory and the almost sure equality 
$
D_2(\mu_\infty)=D_F$ due to  Molchan \cite[Theorem 3]{Mol96} (this almost sure equality is particularly simple in microcanonical cascade case). 
\end{remark*}

\subsection{Organization of the rest part of the paper}

The rest part of the paper is organized as follows:   Section \S \ref{sec-pre} provides the preliminaries on Mandelbrot's microcanonical cascades and Dubins-Freedman random homeomorphisms. Section \S \ref{sec-entropy} develops a new entropy-type inequality for 2D random vectors, while Section \S \ref{sec-PFD} proves polynomial Fourier decay estimates using Pisier's martingale type inequalities and establishes the key lower estimate of the Fourier dimension. Section \S \ref{sec-OP} establishes the optimality of these decay rates through fluctuation analysis of branching random walks. Section \S \ref{sec-main-thm} completes the proof of Theorem \ref{main-thm}, while Section  \S \ref{sec-Holder} addresses H\"older continuity and the proof of Proposition \ref{frostman-measure}.  

\subsection*{Acknowledgements.} 
This work is supported by the NSFC (No.12288201, No.12131016, No. 12201419 and No.12471116). XC is supported by Nation Key R\&D Program of China 2022YFA1006500.

\section{Preliminaries}\label{sec-pre}

\subsection{Mandelbrot's microcanonical  cascades}\label{sec-construct}

The standard dyadic system on the unit interval $[0,1]$ 
is naturally identified with the rooted binary tree 
$\mathcal{T}_2$ (with the root denoted by $\varnothing$) with  
\[
\mathcal{T}_2 =  \{\varnothing\} \sqcup \bigsqcup_{n\ge  1}\{0,1\}^n.
\] 
Then any $u\in \mathcal{T}_2$ can be written as $u= u_1u_2\cdots u_n$ with $u_j\in \{0,1\}$, and  in this case, we  set $|u| = n$ and $u|_k= u_1\cdots u_k$ for $0\le k \le n$ (with convention $u|_0=\varnothing$).      Moreover, we associate $u$ to a dyadic interval $I_u\subset [0,1)$ defined by 
\[
I_u = \Big[\sum_{k=1}^{|u|} u_k 2^{-k},  \sum_{k=1}^{|u|}u_k 2^{-k} +2^{-n} \Big)  \an I_\varnothing = [0,1).
\]

Let $(W(u))_{u\in \mathcal{T}_2}$ be the   i.i.d. copies of a two dimensional random vector $W=(W_0, W_1)$ satisfying  the condition \eqref{def-W0W1}.   For each random vector $W(u)$, write 
\begin{align}\label{def-Wu}
W(u) = (W_0(u), W_1(u)).
\end{align}
For any $n\geq 1$, we define another stochastic process $(X(u))_{u\in \mathcal{T}_2 \setminus \{\varnothing\}}$ indexed by $ \mathcal{T}_2 \setminus \{\varnothing\}$ as follows (see Figure \ref{fig-Xu} for an illustration): 
if $u= u_1\cdots u_n\in \{0,1\}^n$, then 
\begin{align}\label{def-Xu}
X(u):  = 2   W_{u_n} (u_1\cdots u_{n-1}). 
\end{align}
In particular,  the random variable $X(u|_j)$ is given by 
\[
X(u|_j) = 2  W_{u_j} (u|_{j-1}). 
\] 
 For any $n\ge 1$,  define the random   probability measure $\mu_n$  as follows:
 \begin{align}\label{def-mun}
	\mu_n(dt)= \sum_{|u|=n}   \prod_{j=1}^{n}        X(u|_j)     \cdot     \mathds{1}_{I_u}(t)dt,
\end{align}
\begin{figure}[H]
\begin{center}
\includegraphics[scale=0.65]{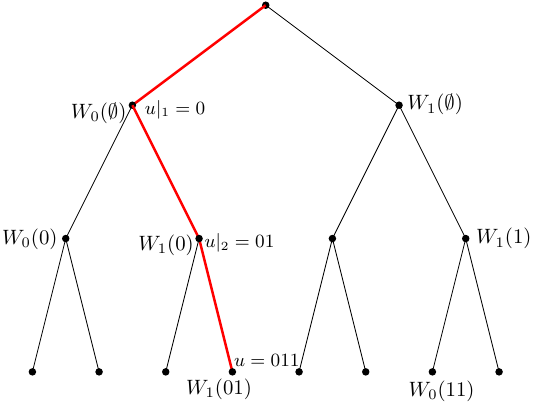}
\end{center}
\caption{An illustration of the stochastic process  $(\frac{X(u)}{2})_{u\in \mathcal{T}_2\setminus \{\varnothing\}}$}\label{fig-Xu}
\end{figure}

By Kahane's fundamental theory of $T$-martingales, almost surely, the random probability  measures $\mu_n$ converge weakly to a limit random  probability measure, denoted by $\mu_\infty$: 
\begin{align}\label{def-mu-inf}
	\mu_n\xrightarrow[n\to\infty]{\text{weakly}} \mu_\infty, \quad a.s. 
\end{align}
The limit random measure $\mu_\infty$ is called the Mandelbrot's microcanonical  cascade measure (also called the microcanonical cascades \cite[p. 311, \S 3.4]{MA-book}) associated to the random vector $W= (W_0, W_1)$.

It is known that, almost surely,  $\mu_\infty([0,1])=1$ and the Hausdorff dimension of $\mu_\infty$ is given by $\dim_H(\mu_\infty) =  -\mathbb{E}[W_0 \log_2 W_0]-\mathbb{E}[W_1 \log_2 W_1]$, see Molchan \cite[Theorem 2]{Mol96}.

\subsection{Dubins-Freedman random homeomorphisms}\label{sec-DF}
The Dubins-Freedman random homeomorphisms (see also Graf, Mauldin and Williams \cite{RA-Homeo}) are defined as follows.

Recall that we denote $(W(u))_{u\in \mathcal{T}_2}$  the i.i.d. copies of a two dimensional random vector $W=(W_0, W_1)$ satisfying the condition \eqref{def-W0W1}.   For each integer $n\ge 1$, define a random step function $\rho_n$   by 
\[
\rho_n (t)  : =  \sum_{|u| = n}  2 W_{u_n}(u_1\cdots u_{n-1}) \cdot  \mathds{1}_{I_u} (t). 
\]
Then, consider the random homeomorphism $F_n$ between $[0,1]$ by 
\begin{align}\label{eqn-connection}
	F_n(t)= \int_0^t f_n(s)ds \quad \text{with\,} f_n(t): = \prod_{j=1}^n \rho_j(t). 
\end{align}
\begin{figure}[H]
	\begin{center}
		\includegraphics[scale=0.65]{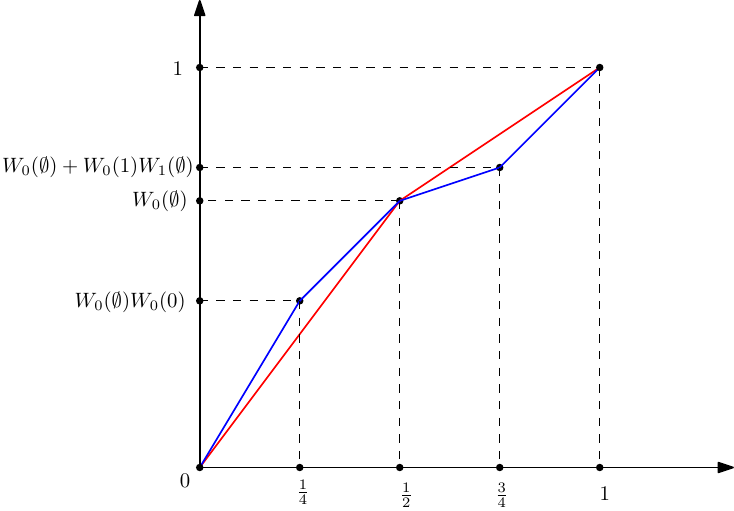}
	\end{center}
	\caption{The first two constructions $F_1$  (the red one) and $F_2$ (the blue one)}\label{fig-rhomeo}
\end{figure}
 As a consequence of the main result in  \cite[Theorem 2.6]{RA-Homeo}, almost surely, $F_n$ converges uniformly to a random homeomorphism $F_\infty: [0,1] \rightarrow [0,1]$.
 
\subsection{Connections}
The study of  the random homeomorphisms $F_n$ and $F_\infty$  naturally fits into the context of microcanonical Mandelbrot cascades. 
Indeed,  denote the random probability measure $d F_n$ by 
\begin{align}\label{def-mu-n}
	\tilde{\mu}_n (dt)= d F_n (t)= \prod_{j=1}^n \rho_j(t) dt \, \,  \text{for $n\ge 1$}.
\end{align}
By convention,  set $\tilde{\mu}_0(dt)= dt$.  One can verify that for any $n\ge1$ and any $u  = u_1\cdots u_n \in \{0,1\}^n$, 
\begin{align}\label{RN-density}
\prod_{j=1}^n \rho_j(t) \Big|_{I_u}  =  2^n \cdot  W_{u_1}(\varnothing)   W_{u_2}(u_1) W_{u_3}(u_1u_2) \cdots W_{u_n}(u_1u_2\cdots u_{n-1}). 
\end{align}
Comparing with \eqref{def-mun}, we get that $\mu_n$ defined in \eqref{def-mun} is nothing but $\tilde{\mu}_n=dF_n$ defined in \eqref{def-mu-n} as above: 
\[
\mu_n= \tilde{\mu}_n = dF_n.
\]
 Hence Mandelbrot's microcanonical cascade  measure $\mu_\infty$ coincides with  the random probability measure  induced by the  Dubins-Freedman random homeomorphism
 $F_\infty$. 
 That is, 
 \[
 \mu_\infty(A)   = \int_A dF_\infty \quad \text{for all measurable $A\subset [0,1]$.}
 \]

\subsection{Notation}

Throughout the paper, by writing $A\lesssim_{x,y} B$, we mean there exists a finite constant $C_{x,y}>0$ depending only on $x,y$ such that $A \le C_{x,y}B$. And, by writing  $A \asymp_{x,y} B$, we mean  $A\lesssim_{x,y}B$ and $B\lesssim_{x,y} A$.

By convention, for any sequence $(c_j)_{j\ge 1}$ in $\C$,  we write 
\[
\prod_{j=1}^0 c_j = \prod_{j\in \varnothing} c_j = 1 \an \sum_{j=1}^0 c_j = \sum_{j\in \varnothing} c_j = 0. 
\]

Given any integrable random variable $X$, we shall write $\mathring{X}$ the centering of $X$: 
\begin{align}\label{def-center}
	\mathring{X}: = X-\E[X]. 
\end{align}

We shall also use the natural filtration:  
\begin{align}\label{def-fil-n}
	\mathscr{F}_n   =  \sigma\Big(\Big\{ \rho_k(t):   k\le n \Big\}\Big)=    \sigma\Big(\Big\{ W(u):   |u|\le n-1 \Big\} \Big)  \, \,  \text{for $n\ge 1$}, 
\end{align}
and by convention, $\mathscr{F}_0$ is defined to be the trivial $\sigma$-algebra.  Note that by the relation  \eqref{def-Xu} between $(X(u))_{|u|\ge 1}$ and $(W(u))_{|u|\ge 0}$, one has 
\[
\sigma\Big(\Big\{ X(u):   |u|\le n \Big\} \Big) =  \sigma\Big(\Big\{ W(u):   |u|\le n-1 \Big\} \Big)  \, \,  \text{for $n\ge 1$}. 
\]

\section{A new entropy-type inequality for 2D-random vectors}\label{sec-entropy}
In this section, we always assume that  $V= (V_0,V_1)$ is a random vector in $\R_{+}^2$ with non-negative coordinates such that 
\[
  V_0 + V_1 = 1 \,\,  a.s.
\]
And define for any $p>0$, 
\begin{align}\label{def-KV}
K_V(p) := \log  \big[(\E[V_0^p + V_1^p])^{1/p}\big] = \log \big[( \E[ \| V\|_{\ell^p}^p] )^{1/p}\big] = \log \|V\|_{L^p(\ell^p)},
\end{align}
where as usual, for any vector $x\in \R^d$ and any $p>0$,  we write its $\ell^p$-norm
\[
\|x\|_{\ell^p} = \Big(\sum_{i=1}^d |x_i|^p\Big)^{1/p}.
\]

Note that the $\ell^p$-norm of a given vector is non-increasing on $p$ and the $L^p$-norm of a given random variable is non-decreasing on $p$. Therefore, a priori, it is not clear whether  $K_V(p)$ is monotone or not as a function of $p$.   

For $d = 2$, we have the following unexpected monotonicity  of the function $K_V(p)$ on the interval $[1,2]$. The general situation for $d\ge 3$ is not clear  to the authors at the time of writing. 

\begin{proposition}\label{prop-p-2}
The function $K_V$ is non-increasing on the interval $[1,2]$. Moreover, $K_V$ is  strictly decreasing on $[1,2]$ if $\E[V_0V_1]>0$. 
In particular,  for all $p\in [1,2]$, the following entropy-type inequality holds
\begin{align}\label{ent-ineq}
\E \big[ V_0^p \log ( V_0^p)  + V_1^p \log ( V_1^p)  \big] \le      \E \big[ V_0^p  + V_1^p \big] \log \big( \E \big[   V_0^p + V_1^p  \big]\big).
\end{align}
 Moreover, the equality holds at one point $p \in [1,2]$ if and only if  $V_0V_1 = 0 \,\, a.s.$
\end{proposition}

\begin{remark*}
By numerical experiments, we know that the inequality \eqref{ent-ineq} fails in higher dimension $d \ge 17$. We believe that $17$ is not optimal and it could be that \eqref{ent-ineq} already fails for $d= 3$. However, at the time of writing, we are not able to prove this. 
\end{remark*}

\begin{lemma}\label{lem-3-2}
We have 
\begin{align}\label{3-less-2}
\|V\|_{L^3(\ell^3)}\le \|V\|_{L^2(\ell^2)}, 
\end{align}
with the equality holds if and only if 
$
V_0V_1 =0 \, \, a.s. 
$
\end{lemma}

\begin{proof}
By setting 
\begin{align}\label{def-c}
c =  \E[V_0]- \E[V_0^2] = \E[V_0V_1] \ge 0,
\end{align}
 we obtain 
\[
\|V\|_{L^2(\ell^2)}^2 = \E[V_0^2] + \E[(1-V_0)^2] = 1 - 2 \E[V_0] + 2 \E[V_0^2] = 1 - 2c 
\]
and 
\[
\|V\|_{L^3(\ell^3)}^3=\E[V_0^3]+\E[(1-V_0)^3]=1-3\E[V_0]+3\E[V_0^2] = 1-3c.
\]
Consequently, 
\[
\|V\|_{L^2(\ell^2)}^6  - \|V\|_{L^3(\ell^3)}^6 = (1-2c)^3 - (1- 3c)^2 = c^2(3-8c).
\] 
Using the definition \eqref{def-c} for $c$, we have 
\[
3 - 8c = 3 - 8 \E[V_0] + 8 \E[V_0^2] = 1  + 8 \E[ (V_0-1/2)^2]\ge 1. 
\]
Then the desired inequality \eqref{3-less-2} follows, with the equality holds if and only if $c=0$, which is equivalent to $V_0V_1 =0 \, \, a.s.$
\end{proof}

\begin{proof}[Proof of Proposition \ref{prop-p-2}]
By the standard  complex interpolation method  on $L^p$-spaces (see \cite[Chapter 5, Theorem 5.1.1, p.106]{BL76}),    if  $\theta \in (0,1)$ and   $p_0, p_1, p_\theta \in [1, \infty)$ satisfy 
\[
\frac{1}{p_\theta} = \frac{1 - \theta}{p_0} + \frac{\theta}{p_1},
\]
then 
\[
\|V\|_{L^{p_\theta}(\ell^{p_\theta})} \le  \|V\|_{L^{p_0}(\ell^{p_0})}^{1-\theta}  \|V\|_{L^{p_1}(\ell^{p_1})}^{\theta}. 
\]
Therefore,  by the definition \eqref{def-KV} of the function $K_V$, 
\[
K_V(p_\theta)  \le (1- \theta) K_V(p_0) + \theta K_V(p_1). 
\]
 In other words, the  function $
(0,1] \ni t \mapsto  f_V(t): =  K_V(1/t) $
is convex. 

Lemma \ref{lem-3-2} implies that  $f_V(1/3) \le f_V(1/2)$. Hence,  by the convexity of $f_V$, for any $r, s \in [1,2]$ with $r<s$,  we have $\frac{1}{r}> \frac{1}{s}> \frac{1}{2}> \frac{1}{3}$ and 
\[
\frac{K_V(r)  - K_V(s)}{\frac{1}{r} - \frac{1}{s}} = \frac{f_V(1/r)  - f_V(1/s)}{\frac{1}{r} - \frac{1}{s}} \ge     \frac{f_V(1/2)  - f_V(1/3)}{\frac{1}{2} - \frac{1}{3}}   \ge 0. 
\]
This implies that $K_V$ is non-increasing on  the interval $[1,2]$. Lemma \ref{lem-3-2} also implies that if $\E[V_0V_1] >0$, then $f_V(1/2)> f_V(1/3)$ and hence $K_V$ is decreasing on $[1,2]$. In particular, 
\[
K_V'(p)\le 0 \quad \text{for all $p\in [1,2]$}, 
\]
which implies the desired inequality \eqref{ent-ineq}.  Finally, if $\E[V_0V_1]>0$, then by Lemma \ref{lem-3-2},  for any $p\in [1,2]$, we have $1/p\ge 1/2 >1/3$ and 
\[
   - p^2 K_V'(p)= f_V'(1/p) \ge    \frac{f_V(1/2)  - f_V(1/3)}{\frac{1}{2} - \frac{1}{3}}   > 0. 
\] 
This completes the whole proof. 
\end{proof}

\section{Polynomial Fourier decay}\label{sec-PFD}

This section is devoted to the proof of Proposition \ref{prop-vec}.   Indeed, it suffices to show that, for any $\varepsilon>0$, there exists a large enough $q>2$ such that 
\begin{align}\label{goal-ineq}
\E\Big[ \Big\{\sum_{s  = 1}^\infty \big(s^{\frac{D_F}{2}-\varepsilon}   \cdot |\widehat{\mu}_\infty(s)|\big)^q\Big\}^{2/q} \Big] = \Big\|      \Big(s^{\frac{D_F}{2}-\varepsilon}   \cdot \widehat{\mu}_\infty(s) \Big)_{s\ge 1}   \Big\|_{L^2(\ell^q)}^2<\infty. 
\end{align}

\subsection{The $\ell^q$-vector valued martingale}

Fix any $\alpha$ with 
\[
0<\alpha<1/2.
\] We are going to study the random vectors in $\C^\N$ generated by  the Fourier coefficients of the random cascade probability measure $\mu_\infty$ obtained in \eqref{def-mu-inf}: 
\begin{align}\label{def-M}
M = M^{(\alpha)}: = (s^\alpha \widehat{\mu}_\infty (s))_{s\ge 1} \in \C^\N. 
\end{align}

{\flushleft \bf Alarm: } {\it A priori, we do not know whether,  the random vector $M$   in $\C^\N$ in \eqref{def-M} almost surely represents a vector in $\ell^q$. }

Recall the relation  \eqref{def-Xu} between $(X(u))_{|u|\ge 1}$ and $(W(u))_{|u|\ge 0}$.  Recall the increasing filtration  $(\mathscr{F}_n)_{n\ge 0}$ of $\sigma$-algebras introduced in \eqref{def-fil-n}: 
\begin{align}\label{def-fil-n-bis}
\mathscr{F}_n   =    \sigma\Big(\Big\{ X(u):   |u|\le n \Big\} \Big) =     \sigma\Big(\Big\{ W(u):   |u|\le n-1 \Big\} \Big)  \, \,  \text{for $n\ge 1$}, 
\end{align}
and by convention, $\mathscr{F}_0$ is defined to be the trivial $\sigma$-algebra.   Recall also the definition of the random measures $\mu_n$ given in \eqref{def-mu-n} and \eqref{def-mun}.  

\begin{definition*}
For any integer $n\ge 0$, define a random vector $M_n = (M_n(s))_{s\ge 1} \in \C^\N$ by 
\begin{align}\label{def-Mn}
M_n(s):= \E[M(s) |\mathscr{F}_n] = s^\alpha\widehat{\mu}_n(s) = s^\alpha \int_{[0,1]} e^{2\pi i s t} d\mu_n(t).
\end{align}
Note that $M_0(s) \equiv 0$ for all $s\ge 1$. 
\end{definition*}

Now, by  Lemma \ref{lem-lq-mart} below, we see that $(M_n)_{n\ge 1}$ is an $\ell^q$-vector-valued martingale with finite second moment $\E[\|M_n\|_{\ell^q}^2]<\infty$ for each $n$. However,  the very rough estimate in the proof of Lemma \ref{lem-lq-mart} does not yield the desired uniform $L^2$-boundedness of the martingale $(M_n)_{n\ge 1}$. Indeed, the uniform $L^2$-boundedness  of $(M_n)_{n\ge 1}$  is given in \S \ref{sec-unif}, where Pisier's martingale type inequalities play a key role and are applied in two different places in the proof.

\begin{lemma}[A very rough estimate]\label{lem-lq-mart}
For any $n\ge 0$ and any  $q> \frac{1}{1-\alpha}>2$,  we have 
\[
\E[\| M_n\|_{\ell^q}^2]<\infty.
\]
Thus, $(M_n)_{n\ge 0}$ is an $\ell^q$-vector-valued martingale with respect to the filtration $(\mathscr{F}_n)_{n
\ge 0}$. 
\end{lemma}

\begin{proof}
Recall that  the random measure $\mu_n (dt)= f_n(t)dt$ has  a random  density $f_n$ (see \eqref{eqn-connection} and \eqref{RN-density}) and $f_n$ is  constant on each dyadic interval $I_u$ with $|u|=n$. By writing  
\[
 f_n(t)|_{I_u}: = \underbrace{2^n \cdot  W_{u_1}(\varnothing)   W_{u_2}(u_1) W_{u_3}(u_1u_2) \cdots W_{u_n}(u_1u_2\cdots u_{n-1})}_{\text{denoted $R_u$}},
\]
 we have, for any integer $s\ge 1$, 
\[
\widehat{\mu}_n(s)=  \sum_{|u| = n} \int_{I_u} f_n(t) e^{ 2\pi  i st} dt= \sum_{|u|=n} R_u \int_{I_u} e^{2\pi i st} dt.
\]
Since  $0<W_{u_k}(u_1\cdots u_{k-1})<1$ for all $u$ and $k$,  we have $0\le R_u \le 2^n$.  Hence for any integer $s\ge 1$
\[
|\widehat{\mu}_n(s)|\le 2^n \sum_{|u|=n}  \Big| \int_{I_u}e^{2\pi i s t}dt \Big| \le  \frac{4^n}{\pi s}.  
\]
Therefore, by the assumption $q(1-\alpha)>1$ and the following inequality 
\[
\sum_{s\ge 1} |s^\alpha \widehat{\mu}_n(s)|^q \le   \frac{4^n}{\pi^q} \sum_{s\ge1}   s^{-(1-\alpha)q}, 
\]
we obtain $\|M_n\|_{L^\infty(\ell^q)}<\infty$.  The desired inequality follows immediately. 
\end{proof}

\subsection{Uniform $L^2$-boundedness of $(M_n)_{n\ge 0}$ via Pisier's martingale type inequalities}\label{sec-unif}
To obtain Proposition~\ref{prop-vec}, we need to prove the uniform $L^2$-boundedness of the $\ell^q$-vector-valued martingale $(M_n)_{n\ge 0}$  for very large $q$ (see Lemma~\ref{lem-find-q} below for the choice of $q$): 
\begin{align}\label{L2-unif-bdd}
\sup_{n\ge 0} \E[\|M_n\|_{\ell^q}^2]<\infty. 
\end{align}

The key ingredient in our proof of  the inequality \eqref{L2-unif-bdd}   is twice crucial applications of  the {\it martingale type-$2$ inequality} of the Banach space $\ell^q$ for $q\ge 2$ (see \cite[p. 409, Definition 10.41]{Pisier-book}):  there exists a constant $C_q>0$ such that for any $\ell^q$-vector-valued martingale  
$(Z_m)_{m\ge 0}$ in $L^2(\PP; \ell^q)$, 
\begin{align}\label{def-Mtype}
    \E [ \| Z_n\|_{\ell^q}^2] \le  C_q \sum_{k =0}^n  \E[ \|Z_k- Z_{k-1}\|_{\ell^q}^p]
\end{align}
with the convention $Z_{-1}\equiv 0$.  In particular, the inequality \eqref{def-Mtype} implies that,  for any family of independent and  centered  random variables $(\Delta_k)_{k=0}^m$  in $L^2(\PP; \ell^q)$, 
\begin{align}\label{def-ind-Mtype}
\E\Big[ \Big\|\sum_{k=0}^m \Delta_k \Big\|_{\ell^q}^2 \Big] \le C_q   \sum_{k =0}^m \E[\|\Delta_k \|_{\ell^q}^2].  
\end{align}

The proof of the inequality \eqref{L2-unif-bdd} is outlined as follows.     In particular,  we indicate the two places where Pisier's martingale type inequalities are used. 
\begin{itemize}
\item The first application of martingale type-$2$ inequality:  For applying martingale type inequality \eqref{def-Mtype} to our $\ell^q$-vector-valued martingale $(M_n)_{n\ge 0}$ introduced in \eqref{def-Mn},  we  first define the sequence of   the martingale differences  $(D_m)_{m\ge 1}$: 
\begin{align}\label{def-Dm}
D_m(s) := M_m(s) - M_{m-1}(s)  \quad \text{for all $m\ge 1$ and $s\ge 1$}. 
\end{align}
Note that, we have $M_0\equiv 0$. Hence, by \eqref{def-Mtype}, we get 
\begin{align}\label{1-type}
\E[\| M_n\|_{\ell^q}^2] \le  C_q \sum_{m =1}^n  \E[ \|D_m\|_{\ell^q}^p]. 
\end{align}
\item The second application of martingale type-$2$ inequality:  for each $1\le m\le n$, we find that (see Lemma \ref{lem-MD}), each martingale difference $D_m$ can be decomposed as the following summation 
\[
D_m = \sum_{|u|=m-1}    \Delta_u, 
\]
where  $\Delta_u$ are  random vectors in $\ell^q$ with explicit form (see \eqref{eqn-m-difference-bis} below). From the explicit forms of all the random vectors $\Delta_u$, one immediately sees that, conditioned on $\mathscr{F}_{m-1}$, they are independent and satisfy  $\E[\Delta_u|\mathscr{F}_{m-1}] =0$. Consequently, we may apply the conditional version of \eqref{def-ind-Mtype} and obtain 
\[
\E\big[   \| D_m\|_{\ell^q}^2 \big|\mathscr{F}_{m-1}\big] \le C_q  \sum_{|u|=m-1}\E\big[  \| \Delta_u\|_{\ell^q}^2 \big|\mathscr{F}_{m-1}\big]. 
\]
Therefore, by taking expectation on both sides, we obtain 
\begin{align}\label{cond-type}
\E\big[   \| D_m\|_{\ell^q}^2 \big] \le C_q  \sum_{|u|=m-1}\E\big[  \| \Delta_u\|_{\ell^q}^2\big]. 
\end{align}
\item Combining the inequalities \eqref{1-type} and \eqref{cond-type}, we obtain 
\[
\E[\| M_n\|_{\ell^q}^2] \le  C_q^2 \cdot  \sum_{m =1}^n \sum_{|u|=m-1}\E\big[  \| \Delta_u\|_{\ell^q}^2\big].
\]
\item For each $1\le m\le n$ and $|u|=m-1$, it turns out that $\E\big[  \| \Delta_u\|_{\ell^q}^2\big]$ has very simple form and can be effectively estimated from above. 
\end{itemize}

Now we proceed to the proof of the main inequality \eqref{L2-unif-bdd}. 

We start with introducing some notations.  Recall the stochastic process   $(X(u))_{u\in \mathcal{T}_2\setminus \{\varnothing\}}$ defined in \eqref{def-Xu}. 
 Using the notation \eqref{def-center}, in what follows, we denote 
\[
\mathring{X}(u) = X(u) - \E[X(u)] = X(u)-1.
\]
 We shall denote the left end-point of the dyadic interval $I_u$ by $\ell_u$. That is, 
 \begin{align}\label{def-lu}
 \ell_u := \sum_{k=1}^{|u|} u_k 2^{-k} \an \ell_\varnothing = 0. 
 \end{align}
 It will be convenient for us to denote, for any integers $m, s \ge 1$
 \begin{align}\label{def-Kms}
  \kappa_m(s):= \frac{e^{i2\pi s2^{-m}}-1}{i2\pi s }.
\end{align}
And, for any $s, m\ge 1$ and $|u|=m-1$, set 
\begin{align}\label{eqn-TuSM}
T(u,s,m): =    \mathring{X}(u0) + e^{i2\pi s 2^{-m}} \mathring{X}(u1) =   2\mathring{W}_0(u)+2 e^{i2\pi s2^{-m}} \mathring{W}_1(u) . 
\end{align}
It is important for our purpose that, for fixed $s, m\ge 1$, conditioned on $\mathscr{F}_{m-1}$, the family
\[
\big\{T(u, s, m)\big\}_{|u| = m-1}
\]
are conditionally centered and independent, hence for distinct $u\ne u'$ with $|u|=|u'|=m-1$, 
\begin{align}\label{cond-ortho}
\E \big[T(u, s, m) \overline{T(u', s, m)}\big| \mathscr{F}_{m-1}\big]= 0. 
\end{align}

The martingale differences $D_m$ defined in \eqref{def-Dm} have the following explicit form. Recall that, since $M_0(s) \equiv 0$ for all $s\ge 1$, by an elementary computation, we have 
\begin{align}\label{form-D1}
D_1 (s)= M_1(s) =       
\left\{ \begin{array}{cl}
\frac{2i}{\pi } \cdot  s^{\alpha-1} \cdot (W_0-W_1)  & \text{if $s$ is odd}
\\
0 & \text{if $s$ is even}
\end{array}
\right..
\end{align}
\begin{lemma}\label{lem-MD}
For any $m\ge 2$ and $s\ge 1$,  the martingale difference  $D_m(s)$ is given by
\begin{align}\label{eqn-m-difference}
D_m(s) = \sum_{|u|=m-1}    \Delta_u(s), 
\end{align}
with $\Delta_u(s)$ defined as
\begin{align}\label{eqn-m-difference-bis}
\Delta_u(s) = s^\alpha\kappa_m(s)  e^{i2\pi s \ell_u }\Big(\prod_{j=1}^{m-1} X(u|_j)\Big)   T(u, s,m).
\end{align}
\end{lemma}

\begin{proof}
 Note that for any $|u|=m$, by the definition \eqref{def-Kms} of  $\kappa_m(s)$, 
\begin{align}\label{kappa-m-int}
\int_{I_u} e^{i2 \pi s x}dx = \kappa_m(s) e^{i2\pi s \ell_u}.
\end{align}
 By \eqref{def-mun}, for any integer $s\ge 1$, 
\begin{align*}
\widehat{\mu}_m(s)& =\int_{0}^1 e^{i2\pi s x}d\mu_m(x) =   \kappa_m(s) \cdot \sum_{|u|=m} e^{i2\pi s \ell_u} \prod_{j=1}^m X(u|_j). 
\end{align*}
Thus, by using the equality  
\[
\frac{\kappa_{m-1}(s)}{\kappa_m(s)}=1 + e^{i2\pi s2^{-m}},
\]
 we obtain 
\begin{align*}
\widehat{\mu}_{m-1}(s) = \kappa_{m}(s) \cdot \sum_{|v|=m-1} e^{i2\pi s \ell_v} \Big( \prod_{j=1}^{m-1} X(v|_j) \Big)  \cdot (1 +  e^{i2\pi s2^{-m}}) . 
\end{align*}
Now, for  each $u$ with $|u|=m$, we may write it as $u= v u_m$ with $v= u|_{m-1}$. Then  using
\[
\ell_u = \ell_v + u_m 2^{-m}  \an u|_j = v|_j \quad \text{for all $j\le m-1$,}
\]
 we obtain 
\begin{align*}
\widehat{\mu}_m(s)  & =\kappa_m(s)  \sum_{|v|=m-1}    \Big[  e^{i2\pi s  \ell_v } \Big(\prod_{j=1}^{m-1}  X(v|_j)\Big) X(v0) + e^{i2\pi s  \ell_v }   e^{i2\pi s 2^{-m}} \Big(\prod_{j=1}^{m-1}  X(v|_j)\Big) X(v1)  \Big]
\\
& = \kappa_m(s)  \sum_{|v|=m-1}       e^{i2\pi s  \ell_v } \Big(\prod_{j=1}^{m-1}  X(v|_j)\Big) \cdot  \Big( X(v0) + e^{i2\pi s 2^{-m} }  X(v1)  \Big).
\end{align*}
Consequently,  by recalling  $\mathring{X}(u) = X(u)-1$, we obtain 
\begin{align*}
D_m(s) = & \widehat{\mu}_m(s) - \widehat{\mu}_{m-1}(s)
\\
 =& \kappa_m(s) \cdot \sum_{|v|=m-1}       e^{i2\pi s  \ell_v } \Big(\prod_{j=1}^{m-1}  X(v|_j)\Big) \cdot  \Big( \mathring{X}(v0) + e^{i2\pi s 2^{-m} }  \mathring{X}(v1)  \Big).
\end{align*}
This completes the proof of the lemma. 
\end{proof}

\begin{lemma}\label{lem-T}
For any $q\geq 2$, 
\begin{align}\label{eqn-tusms}
\E[|T(u,s,m)|^q]=2^q|1-e^{i2\pi s2^{-m}}|^q \cdot \E[|\mathring{W}_0|^q]. 
\end{align}
Moreover, 
\begin{align}\label{T-square}
\begin{split}
\E[|T(u,s,m)|^2] &= 4 |1 - e^{i2\pi s2^{-m}}|^2  \Var(W_0);
\\
\E[T(u,s,m)^2] &= 4 (1 - e^{i2\pi s2^{-m}})^2  \Var(W_0).
\end{split}
\end{align}
\end{lemma}
\begin{proof}
By $W_1=1-W_0$, we have  $\mathring{W}_0  =- \mathring{W}_1$ and thus 
\[
T(u, s, m) = 2 (1 - e^{i2\pi s2^{-m}})   \mathring{W}_0(u). 
\]
Lemma \ref{lem-T} follows immediately. 
\end{proof}

Recall the definition \eqref{def-DF} of $D_F \in (0,1)$:
\[
D_F= \log_2 \Big( \frac{1}{\E[W_0^2]+\E[W_1^2]}\Big).
\]
Clearly, we have 

\begin{lemma}\label{lem-find-q}
 Let $\alpha \in (0,D_F/2)$.  Then for any $q>\frac{2}{D_F-2 \alpha}$, we have 
\[
q>\frac{1}{1-\alpha}>2
\an 
(\E[W_0^2]+\E[W_1^2]) \cdot 2^{2\alpha+\frac{2}{q}}<1. 
\] 
\end{lemma}

\begin{proof}[Proof of Proposition~\ref{prop-vec}]
Fix any $\alpha \in (0,D_F/2)$ and take any $q>\frac{2}{D_F-2 \alpha}$.  By Lemma~\ref{lem-find-q}, we have $q>2$ and hence  the Banach space $\ell^q$ has  martingale type-$2$ (see  \cite[p. 409, Definition 10.41]{Pisier-book} for its precise definition). Consequently,   for any $n\ge 1$, we get 
\begin{align*}
\| (s^\alpha \widehat{\mu}_\infty(s))_{s\geq 1} \|_{L^2(\ell^q)}^2 \lesssim_q \sum_{m=1}^\infty \|D_m\|_{L^2(\ell^q)}^2,
\end{align*}
with the martingale differences $D_m$ defined as in \eqref{def-Dm}. 

Notice that, by the explicit form  \eqref{eqn-m-difference} and \eqref{eqn-m-difference-bis} for  $D_m$, conditioned on  $\mathscr{F}_{m-1}$, the martingale difference $D_m$  is the sum of independent centered random vectors in $\ell^q$.  Therefore, by applying again the martingale type-$2$ property of  $\ell^q$ and recalling the notation \eqref{eqn-TuSM},  we get 
\begin{align*}
\E_{m-1}[\|D_m\|_{\ell^q}^2]\lesssim_q &\sum_{|u|=m-1} \E_{m-1}\Big[\Big\|s^\alpha\kappa_m(s)e^{i2\pi s \ell_u} \Big( \prod_{j=1}^{m-1} X(u|_j)\Big) T(u,s,m)\Big\|_{\ell^q}^2\Big]\\
=&\sum_{|u|=m-1} \Big( \prod_{j=1}^{m-1} X(u|_j)^2\Big) \cdot  \E\Big[\Big\{ \sum_{s=1}^\infty |s^\alpha\kappa_m(s)|^q\cdot |T(u,s,m)|^q\Big\}^{2/q} \Big].
\end{align*}
Observe that $2/q\le 1$,  by Jensen's inequality, we obtain  
\[
\E\Big[\Big\{ \sum_{s=1}^\infty |s^\alpha\kappa_m(s)|^q\cdot |T(u,s,m)|^q\Big\}^{2/q} \Big] \le  \Big\{ \sum_{s=1}^\infty |s^\alpha\kappa_m(s)|^q \cdot \E\big[|T(u,s,m)|^q\big]\Big\}^{2/q}.
\]
It follows that, 

\begin{align*}
\E_{m-1}[\|D_m\|_{\ell^q}^2]\lesssim_q  & \sum_{|u|=m-1} \Big (\prod_{j=1}^{m-1} X(u|_j)^2 \Big)\cdot \Big\{ \sum_{s=1}^\infty |s^\alpha\kappa_m(s)|^q \cdot \E\big[|T(u,s,m)|^q\big]\Big\}^{2/q},
\end{align*}
The above inequalities combined with \eqref{def-Kms} and \eqref{eqn-tusms} yield
\[
\E_{m-1}[\|D_m\|_{\ell^q}^2]\lesssim_q  \sum_{|u|=m-1}\Big(\prod_{j=1}^{m-1} X(u|_j)^2\Big) \cdot \underbrace{\Big(\sum_{s=1}^\infty \frac{|e^{i2\pi s 2^{-m}}-1|^{2q}}{s^{q (1-\alpha)}} \Big)^{2/q}}_{\text{denoted $U(m,q, \alpha)$}}.
\]
By taking expectations on both sides, one gets 
\[
\E[\|D_m\|_{\ell^q}^2] \lesssim_q \sum_{|u|=m-1}  \Big( \prod_{j=1}^{m-1} \E[X(u|_j)^2] \Big)  \cdot U(m,q,\alpha). 
\]
Note that, by \eqref{def-Xu},
\begin{align}\label{eqn-usefulpoi}
\sum_{|u|=m-1}  \prod_{j=1}^{m-1} \E[X(u|_j)^2]  = \sum_{|u|=m-1}     \prod_{j=1}^{m-1} \E[2^2 W_{u_j}^2]  = 4^{m-1}  (\E[W_0^2]+\E[W_1^2])^{m-1}.  
\end{align}
Hence 
\begin{align*}
\E[\|D_m\|_{\ell^q}^2] \lesssim_q   2^{2m} \cdot (\E[W_0^2]+\E[W_1^2])^{m} \cdot U(m,q,\alpha).
\end{align*}
It follows that  the random vector $(s^\alpha \widehat{\mu}_\infty(s))_{s\geq 1}$ satisfies 
\begin{align}\label{es-U}
\| (s^\alpha \widehat{\mu}_\infty(s))_{s\geq 1} \|_{L^2(\ell^q)}^2 \lesssim_q \sum_{m=1}^\infty 2^{2 m}(\E[W_0^2]+\E[W_1^2])^{m} \cdot  U(m, q, \alpha).
\end{align}

{\flushleft \bf Claim A:}  For any $q, \alpha$ such that $q(1-\alpha)>1$ and $0\le \alpha<D_F/2<1/2$, we have 
\begin{align*}
U(m, q, \alpha)   \lesssim_{q, \alpha}   2^{-2m(1-\alpha - \frac{1}{q})} \quad \text{for all $m\ge 1$}.
\end{align*}

Using \eqref{es-U} and  Claim A, we get
\begin{align*}
\| (s^\alpha \widehat{\mu}_\infty(s))_{s\geq 1} \|_{L^2(\ell^q)}^2  \lesssim_{q, \alpha} & \sum_{m=1}^\infty 2^{2 m}(\E[W_0^2]+\E[W_1^2])^{m} \cdot   2^{-2m(1-\alpha- \frac{1}{q})}
\\ \lesssim_{q, \alpha}&\sum_{m=1}^\infty \Big[  (\E[W_0^2]+\E[W_1^2]) \cdot 2^{2\alpha+\frac{2}{q}}\Big]^m.
\end{align*}
By Lemma~\ref{lem-find-q}, our choice of $\alpha$ and $q$ implies 
\[
(\E[W_0^2]+\E[W_1^2]) \cdot 2^{2\alpha+\frac{2}{q}}<1.
\] 
Therefore, we get the desired inequality 
\[
\| (s^\alpha \widehat{\mu}_\infty(s))_{s\geq 1} \|_{L^2(\ell^q)}<\infty.
\]

It remains to prove Claim A.  Indeed,  there exists an absolute constant $C>1$ such that  for all integers $m, s \ge 1$, 
\[
|e^{i 2\pi s 2^{-m}}-1| \le    C \cdot  \min (1,   \, \,  s \cdot  2^{-m}).
\]
Therefore, using the assumption  that $q(1-\alpha)>1$ and $0\le \alpha<D_F/2<1/2$, we obtain 
\begin{align*}
U(m, q, \alpha) &  \lesssim_{q, \alpha}    \Big( \sum_{s=1}^{2^m}    (s \cdot 2^{-m})^{2q} \cdot  s^{-q(1- \alpha)} + \sum_{s\ge 2^{m}} s^{-q(1- \alpha)}  \Big)^{2/q}
\\
& \lesssim_{q, \alpha}     \Big(    2^{-2mq}  \cdot  (2^m)^{2q-q(1- \alpha) + 1} +  (2^m)^{-q(1- \alpha)+1}  \Big)^{2/q}
\\
& \lesssim_{q, \alpha}   2^{-2m(1-\alpha- \frac{1}{q})}.
\end{align*}
This completes the proof of the Claim A and hence the whole proof of Proposition~\ref{prop-vec}.  
\end{proof}

\section{Optimality of the polynomial exponent}\label{sec-OP}

This section is devoted to the proof of Proposition \ref{prop-convergence-in-law} on the fluctuation of the  rescaled  Fourier coefficients $\widehat{\mu}_{\infty}(2^n)$. 

\subsection{Basic properties of the Fourier coefficients}

Note that, since $\E[\mu_\infty(dt)] =dt$ the Lebesgue measure on $[0,1]$,   one has 
\[
\E[\widehat{\mu}_\infty(s)]=0 \quad \text{for any integer $s\ge 1$.}
\]

\begin{lemma}\label{lem-12-moments}
For any integer $s\ge 1$, one has 
 \begin{align}\label{2-moment-bis}
\E[|\widehat{\mu}_\infty(s)|^2]=\frac{\mathrm{Var}[W_0]}{\pi^2 s^2}\sum_{m=1}^\infty|e^{i2\pi s 2^{-m}}-1|^4 \cdot (8\E[W_0^2])^{m-1}.
\end{align}
In particular, for $s =1$, 
 \begin{align}\label{2-moment}
\varrho:=\E[|\widehat{\mu}_\infty(1)|^2]=\frac{\mathrm{Var}[W_0]}{\pi^2}\sum_{m=1}^\infty|e^{i2\pi 2^{-m}}-1|^4 \cdot (8\E[W_0^2])^{m-1}.
\end{align}
\end{lemma}

\begin{remark*}
Fix any integer $s\ge 1$, since  $\E[W_0^2] \in (0,1)$ and 
\[
|e^{i2\pi s 2^{-m}}-1|^4 \cdot (8\E[W_0^2])^{m-1}  = O \Big(\big(\E[W_0^2]/2\big)^m\Big) \quad \text{as $m\to\infty$},
\]
the series \eqref{2-moment-bis} is convergent. 
\end{remark*}

\begin{lemma}\label{lem-0-av}
One has 
\begin{align}\label{mu-12}
\varpi:=\E[\widehat{\mu}_\infty(1)^2] =  - \frac{16\Var[W_0]}{\pi^2} \Big(1   -2 \E[W_0^2]\Big)\in (-\infty, 0). 
\end{align}
\end{lemma}

\begin{proof}[Proof of Lemma \ref{lem-12-moments}]

Take  $\alpha=0$ in \eqref{def-Mn}. Take any $s\ge 1$.   Since $\E[\widehat{\mu}_\infty(s)] = 0$,  by using the orthogonality of the  martingale differences,  we get 
\[
\E[|\widehat{\mu}_\infty(s)|^2]=\sum_{m=1}^\infty\E[|D_m(s)|^2],
\]
where $D_m(s)$ is defined as in \eqref{def-Dm}.  

For $m = 1$, by  \eqref{form-D1}, we have 
\begin{align}\label{m-is-1}
\E[|D_1(s)|^2] =  \left\{ \begin{array}{cl}
\frac{16}{\pi^2 s^2 } \Var[W_0]  & \text{if $s$ is odd}
\\
0 & \text{if $s$ is even}
\end{array}
\right..
\end{align}

For the integers $m\ge 2$, using the explicit form \eqref{eqn-m-difference} and \eqref{eqn-m-difference-bis} of  $D_m(s)$ and the orthogonality \eqref{cond-ortho} of $T(u, m, s)$ conditioned on $\mathscr{F}_{m-1}$,  we have
\begin{align*}
\E_{m-1}[|D_m(s)|^2]
=|\kappa_m(s)|^2\sum_{|u|=m-1}\prod_{j=1}^{m-1} X(u|_j)^2 \cdot \E[|T(u, m, s)|^2].
\end{align*}
Hence, by taking expectation on both sides, then  using \eqref{def-Kms},  Lemma~\ref{lem-T},   \eqref{eqn-usefulpoi} and the elementary equality $\E[W_0^2] = \E[W_1^2]$, we get 
\begin{align}\label{m-g-2}
\begin{split}
\E[|D_m(s)|^2]& = \frac{|e^{i2\pi s2^{-m}}-1|^4}{ \pi^2 s^2 }\mathrm{Var}[W_0] \cdot \sum_{|u|=m-1}\prod_{j=1}^{m-1}\E[X(u|_j)^2] 
\\
&    = \frac{|e^{i2\pi s2^{-m}}-1|^4}{ \pi^2 s^2 }\mathrm{Var}[W_0] \cdot   (8\E[W_0^2])^{m-1} .
\end{split}
\end{align}

Comparing \eqref{m-is-1} and \eqref{m-g-2}, we see that the equality \eqref{m-g-2}  holds for all integers $m\ge 1$. The desired equality \eqref{2-moment-bis} follows immediately. 
\end{proof}

\begin{proof}[Proof of Lemma \ref{lem-0-av}]
  Recall that, if  $(d_n)_{n\ge 1}$ is  any sequence of martingale differences, then for any integers $n\ge m\ge 1$,
\[
\E[d_n d_m] = \E[d_n\bar{d}_m] = 0.  
\]
 Therefore,  by using $D_m(1)$ defined as in \eqref{def-Dm} (here we take $\alpha =0$ and $s=1$), we have 
\begin{align*}
\E[\widehat{\mu}_\infty(1)^2]=\sum_{m=1}^\infty\E[D_m(1)^2]. 
\end{align*}
For $m=1$, by \eqref{form-D1},  we have 
\[
\E[D_1(1)^2] =  
-\frac{16}{\pi^2 } \Var[W_0]. 
\]
For  $m\ge 2$,  using the form \eqref{eqn-m-difference} and \eqref{eqn-m-difference-bis} for $D_m(1)$ (again take $\alpha =0$ and $s=1$), we get 
\[
\E_{m-1}[D_m(1)^2] = \kappa_m(1)^2 \sum_{|u|=m-1} e^{i4\pi  \ell_u }\Big(\prod_{j=1}^{m-1}  X(u|_j)^2 \Big) \E[T(u,m,1)^2]. 
\] 
Then taking expectation on both sides and using  \eqref{T-square}, we obtain 
\begin{align}\label{dm-2}
\E[D_m(1)^2]   = \kappa_m(1)^2 \cdot 4 (1 - e^{i2\pi 2^{-m}})^2  \Var(W_0) \cdot  \sum_{|u|=m-1} e^{i4\pi  \ell_u }\Big(\prod_{j=1}^{m-1}   \E[X(u|_j)^2] \Big). 
\end{align}
By \eqref{def-Xu} and the elementary equality $\E[W_0^2] = \E[W_1^2]$, we get 
\begin{align}\label{simple-id}
 \sum_{|u|=m-1} e^{i4\pi  \ell_u }\Big(\prod_{j=1}^{m-1}   \E[X(u|_j)^2] \Big) = (4 \E[W_0^2])^{m-1}  \sum_{|u|=m-1} e^{i4\pi  \ell_u }. 
\end{align}
By using \eqref{def-lu}, we have 
\[
 \sum_{|u|=m-1} e^{i4\pi  \ell_u }  =   \sum_{u_1,\cdots,u_{m-1}\in\{0, 1\}}      \prod_{j=1}^{m-1}  e^{i4\pi  u_j 2^{-j} }   = \prod_{j=1}^{m-1} (1 + e^{i4 \pi  2^{-j}}). 
 \]
 Observe that for $j=2$, we have $1+e^{i4\pi 2^{-j}}= 1 + e^{i\pi} =0$. Hence 
 \begin{align}\label{c-id}
 \sum_{|u|=m-1} e^{i4\pi  \ell_u }  = 
 \left\{
 \begin{array}{cl}
 2 & \text{if $m = 2$}
 \\
 0 & \text{if $m\ge 3$}
 \end{array}
 \right..
 \end{align}
Combining \eqref{dm-2}, \eqref{simple-id} and \eqref{c-id}, we get 
\[
\E[D_m(1)^2]  = 
\left\{ \begin{array}{cl}
  \frac{32}{\pi^2} \cdot \Var[W_0] \cdot  \E[W_0^2]  & \text{if $m = 2$}
 \\
 0 & \text{if $m\ge 3$}
 \end{array}
 \right.. 
\]
Therefore, we obtain  the desired equality \eqref{mu-12}. 
\end{proof}

\subsection{Basic properties on $\widehat{\mu}_\infty(2^n)$}

Recall again the filtration $(\mathscr{F}_n)_{n\ge 0}$ in \eqref{def-fil-n}: 
\[
\mathscr{F}_n   = \sigma\Big (\Big\{X(u):  |u|\le n \Big\} \Big)=  \sigma\Big(\Big\{ W(u):   |u|\le n-1 \Big\} \Big)  \, \,  \text{for $n\ge 1$}. 
\]

\begin{lemma}\label{lem-self-similar}
For any $n\geq 1$, we have
\begin{align}\label{eqn-function-equality}
\widehat{\mu}_\infty(2^n)\overset{d}{=}\frac{1}{2^n}\sum_{|u|=n} \Big(\prod_{j=1}^{n} X(u|_j)\Big)\widehat{\mu}_\infty^{(u)}(1),
\end{align}
where $\widehat{\mu}_\infty^{(u)}(1)$ are i.i.d. copies of $\widehat{\mu}_{\infty}(1)$, which are independent of $\mathscr{F}_{n}$.
\end{lemma}

\begin{proof}
Fix any integer $n\ge1$. By  \eqref{def-mun}, for any $k\ge 1$, we have 
\[
\widehat{\mu}_{n+k}(2^n) = \sum_{|u|=n+k}   \prod_{j=1}^{n+k}        X(u|_j)     \cdot     \int_{I_u}  e^{i 2 \pi 2^n x}dx.
\]
Now for each $u$ with $|u|=n+k$, by writing $u$ as $u = vw$ with $|v|=n$ and $|w|= k$, we have 
\begin{align*}
\widehat{\mu}_{n+k}(2^n) = \sum_{|v|=n} \prod_{j=1}^{n}        X(v|_j)     \cdot     \sum_{|w|=k}   \prod_{l=1}^{k}        X(v \cdot w|_l)     \cdot   \int_{I_{vw}}  e^{i 2 \pi 2^n x}dx.
\end{align*}
Observe that  for $|v|=n$ and $|w|= k$, 
\begin{align*}
 \int_{I_{vw}}  e^{i 2 \pi 2^n x}dx & =  \frac{1}{i 2\pi 2^n} \cdot \exp\Bigg(i 2\pi 2^n\Big[\sum_{j=1}^n v_j 2^{-j} + \sum_{l=1}^k w_l 2^{-n-l}\Big]\Bigg) \cdot\Big(  e^{i 2 \pi 2^n 2^{-n-k}} -1\Big)
 \\
 & =  \frac{1}{i 2\pi 2^n} \cdot \exp\Big(i 2\pi   \sum_{l=1}^k w_l 2^{-l}\Big) \cdot\Big(  e^{i 2 \pi 2^{-k}} -1\Big)
 \\
 & = \frac{1}{2^n}\int_{I_w}  e^{i 2\pi x} dx.
\end{align*}
It follows that 
\[
\widehat{\mu}_{n+k}(2^n) =  \frac{1}{2^n}\sum_{|v|=n} \prod_{j=1}^{n}        X(v|_j)     \cdot     \sum_{|w|=k}   \prod_{l=1}^{k}        X(v \cdot w|_l)     \cdot   \int_{I_{w}}  e^{i 2 \pi x}dx.
\]
By letting $k\to\infty$, we obtain the desired equality \eqref{eqn-function-equality}. 
\end{proof}

Recall that for a complex random variable $X+iY$,  we denote by 
\[
\mathrm{Cov}(X+iY): =      \begin{pmatrix}
\Var(X) & \Cov(X, Y)\\
\Cov(X, Y) &\Var(Y) \\
\end{pmatrix}.
\]
In other words,  $\mathrm{Cov}(X+iY)$ denotes the covariance matrix of the real random vector $(X, Y)$.  
Define the following non-negative martingale with respect to the filtration $(\mathscr{F}_n)_{n\ge 1}$: 
\begin{align}\label{eqn-def-mart}
\mathcal{M}_n^{(2)}=\frac{1}{2^n}\sum_{|u|=n}\prod_{j=1}^{n} \frac{X(u|_j)^2}{\E[4W_0^2]}=\frac{1}{(8\E[W_0^2])^n}\sum_{|u|=n}\prod_{j=1}^{n} X(u|_j)^2,\quad n\geq 1.
\end{align}
Recall the definition \eqref{def-DF} of $D_F \in (0,1)$,  the definition \eqref{2-moment} of  $\varrho=\E[|\widehat{\mu}_\infty(1)|^2]$  and the definition \eqref{mu-12} of  $\varpi=\E[(\widehat{\mu}_\infty(1))^2]$.

\begin{lemma}\label{lemma-mean-zero-s}
We have 
\[
\E\Big[ 2^{n D_F}|\widehat{\mu}_{\infty}(2^n)|^2 \big| \mathscr{F}_{n}\Big] =\varrho\mathcal{M}_n^{(2)}.
\]
Moreover,
\[
\E\Big[2^{n D_F}(\widehat{\mu}_{\infty}(2^n))^2 \big| \mathscr{F}_{n}\Big] =\varpi\mathcal{M}_n^{(2)}.
\]
\end{lemma}

 Notice that $|\varpi|<\varrho$, hence $\varrho\pm\varpi>0$. Lemma~\ref{lemma-mean-zero-s} immediately implies the following
\begin{corollary} 
Let $n\ge 1$ be an integer.  Conditioned on $\mathscr{F}_{n}$,   the covariance matrix of the the complex random variable $2^{\frac{n D_F}{2}}\cdot \widehat{\mu}_{\infty}(2^n)$ is   given by 
\begin{align}\label{eqn-cov-mat}
\mathrm{Cov}\big[2^{\frac{nD_F}{2}}\widehat{\mu}_\infty(2^n)\mid\mathscr{F}_n\big]=\frac{1}{2}\mathcal{M}_n^{(2)}\begin{pmatrix}
\varrho+\varpi&0\\
0 &\varrho-\varpi\\
\end{pmatrix}.
\end{align}
In particular, 
\begin{align*}
\E\big[2^{nD_F}|\widehat{\mu}_\infty(2^n)|^2\big]=\varrho \an 
\E\big[2^{nD_F}(\widehat{\mu}_\infty(2^n))^2\big]=\varpi
\end{align*}
and 
\begin{align*}
\mathrm{Cov}\big[2^{\frac{nD_F}{2}}\widehat{\mu}_\infty(2^n)\big]=\frac{1}{2}\begin{pmatrix}
\varrho+\varpi& 0 \\
0 &\varrho-\varpi\\
\end{pmatrix}.
\end{align*}
\end{corollary}

\begin{proof}[Proof of Lemma \ref{lemma-mean-zero-s}]
By the definition \eqref{def-DF} of $D_F$,  one has $2^{D_F}=\frac{1}{2\E[W_0^2]}$.
Hence by \eqref{eqn-function-equality}, 
\begin{align}\label{eqn-alpha-equiv-s-s}
2^{\frac{n D_F}{2}} \widehat{\mu}_{\infty}(2^n)=\frac{1}{(8\E[W_0^2])^{\frac{n}{2}}}\sum_{|u|=n} \Big(\prod_{j=1}^{n} X(u|_j)\Big) \widehat{\mu}_\infty^{(u)}(1) .
\end{align}
Note that for any  $u\ne v$ with $|u|=|v|=n$,
\[
\E\Big [\widehat{\mu}_{\infty}^{(u)}(1)   \cdot  \overline{\widehat{\mu}_{\infty}^{(v)}(1)}\big| \mathscr{F}_n\Big] = 0 \an \E\Big [|\widehat{\mu}_{\infty}^{(u)}(1) |^2 \big| \mathscr{F}_n\Big]  =  \E\big[|\widehat{\mu}_{\infty}(1)|^2\big].
\]
Hence
\begin{align*}
\E\Big[ 2^{n D_F}|\widehat{\mu}_{\infty}(2^n)|^2 \big| \mathscr{F}_{n}\Big]
=\frac{ \E\big[|\widehat{\mu}_{\infty}(1)|^2\big]}{(8\E[W_0^2])^{n}}\sum_{|u|=n} \prod_{j=1}^{n} X(u|_j)^2=\varrho\mathcal{M}_n^{(2)}.
\end{align*}
Note also that for any  $u\ne v$ with $|u|=|v|=n$,
\[
\E \big[\widehat{\mu}_{\infty}^{(u)}(1) \widehat{\mu}_{\infty}^{(v)}(1)| \mathscr{F}_n\big] = 0 \an \E\Big [(\widehat{\mu}_{\infty}^{(u)}(1) )^2 \big| \mathscr{F}_n\Big]  =  \E\big[(\widehat{\mu}_{\infty}(1))^2\big].
\]
Hence by \eqref{mu-12} and \eqref{eqn-alpha-equiv-s-s},  we obtain
\begin{align*}
\E\Big[ 2^{n D_F}(\widehat{\mu}_{\infty}(2^n))^2 \big| \mathscr{F}_{n}\Big]
 = \frac{ \E[(\widehat{\mu}_{\infty} (1))^2]}{(8\E[W_0^2])^{n}}\sum_{|u|=n}  \prod_{j=1}^{n} X(u|_j)^2=\varpi\,\mathcal{M}_n^{(2)}.  
\end{align*}
Lemma \ref{lemma-mean-zero-s} is proved. 
\end{proof}

\subsection{Non-vanishing property of the  martingale limit  of  $\mathcal{M}_n^{(2)}$}

Recall the   martingale $\mathcal{M}_n^{(2)}$ defined in \eqref{eqn-def-mart}.    Since $(\mathcal{M}_n^{(2)})_{n\geq 1}$ is a non-negative martingale,  there exists a  random variable $\mathcal{M}_\infty^{(2)} \ge 0$ such that
\[
\mathcal{M}_n^{(2)}\rightarrow \mathcal{M}_\infty^{(2)} \quad a.s.
\]

\begin{lemma}\label{lemma-minfinity-positive}
We have $
\PP(\mathcal{M}_\infty^{(2)}>0)=1$. 
\end{lemma}

Given any random vector $W=(W_0,W_1)$ in \eqref{def-W0W1},  define 
\begin{align}\label{eqn-varphiwp}
\varphi_W(p):  = \log (\E[W_0^p]+\E[W_1^p]),\quad p\in\mathbb{R},
\end{align}
where we take the convention $\log (+\infty)=+\infty$.
It  can be easily checked that:
\begin{itemize}
\item[(1)] $\varphi_W$ is  strictly convex on $(0,\infty)$ except for the trivial case $W_0=W_1=1/2\,\, a.s.$
\item[(2)] $\varphi_W(1)=0$ and $\varphi_W(p)\leq 0$ for $p\in (1,\infty)$ and $\varphi_W(p)\geq 0$ for $p\in[0,1)$.
\end{itemize}

\begin{proof}[Proof of Lemma \ref{lemma-minfinity-positive}]
We shall use Biggins martingale convergence theorem in the context of branching random walks (see, e.g.,  \cite[Chapter 1]{shi-zhan-book}). 

For this purpose, we write the martingale $\mathcal{M}_n^{(2)}$ in the standard form of additive martingale for branching random walks.  First  for all $|u|=n \ge 1$, we set 
\begin{align}\label{eqn-def-yu}
Y(u):=\frac{1}{2^n}\prod_{j=1}^{n} \frac{X(u|_j)^2}{\E[4W_0^2]}=\frac{1}{(8\E[W_0^2])^n}\prod_{j=1}^{n} X(u|_j)^2
\end{align}
and, by setting $\xi(u)=-2 \log X(u)+2 \log 2$, 
 we  define
\begin{align}\label{eqn-def-vu}
V(u)=\sum_{j=1}^{|u|}\xi(u|_j)= -2 \sum_{j=1}^{|u|}\log X(u|_j)+2n\log 2.
\end{align}
Set also 
\begin{align}\label{def-psi-beta}
\psi(\beta): = \log \E\Big[ \sum_{|u|=1} e^{-\beta V(u)}\Big]. 
\end{align}
Then,  by the definition \eqref{def-Xu} for $X(u)$ and   the definition \eqref{eqn-varphiwp} of the function $\varphi_W$, we have 
\begin{align}\label{phi-psi}
\psi(\beta)=  \log \E [W_0^{2\beta} + W_1^{2\beta}] =  \varphi_W(2\beta). 
\end{align}
In particular, 
\begin{align}\label{psi-1}
\psi(1)= \varphi_W(2). 
\end{align}
Since $\E[W_0^2] = \E[W_1^2]$, we have   
\[
-\log Y(u)=n\log (8\E[W_0^2])-2\sum_{j=1}^{n}\log X(u|_j)=V(u)+n\psi(1).
\]
It follows that,   the    martingale $\mathcal{M}_n^{(2)}$ can be re-written as 
\begin{align*}
\mathcal{M}_n^{(2)}=\sum_{|u|=n}Y(u)=\sum_{|u|=n}e^{-V(u)-n\psi(1)}.
\end{align*}
We shall   apply the Biggins martingale convergence theorem (see, e.g., \cite[Theorem 3.2, p. 21]{shi-zhan-book})  in our setting. Clearly, all the conditions  \cite[Theorem 3.2, p. 21]{shi-zhan-book} are satisfied here:
\[
\psi(0)>0, \, \psi(1)<\infty \an \psi'(1)\in \R. 
\]
 Therefore,     $\mathbb{P}(\mathcal{M}_\infty^{(2)}>0)=1$  if and only if 
\begin{align}\label{equi-cond}
\E[\mathcal{M}_1^{(2)} \log_{+} (\mathcal{M}_1^{(2)})]<\infty \an \psi(1)>\psi'(1). 
\end{align}
It remains to check the condition \eqref{equi-cond}.  First of all,  by \eqref{eqn-def-mart}, 
\[
\mathcal{M}_1^{(2)}  =  \frac{1}{8\E[W_0^2]}\sum_{|u|=1} X(u)^2 = \frac{W_0^2 +W_1^2}{2\E[W_0^2]}.
\]
Since $W_0,W_1\in (0,1)$, the random variable $\mathcal{M}_1^{(2)}$ is bounded. Hence 
\begin{align*}
\E[\mathcal{M}_1^{(2)} \log_{+} (\mathcal{M}_1^{(2)})] <\infty. 
\end{align*}
Secondly,  by the relation \eqref{phi-psi} between $\psi$ and $\varphi_W$,  we have 
\begin{align}\label{psi-p-1}
\psi'(1)= 2 \varphi_W'(2).  
\end{align}
 By the proof of Proposition~\ref{prop-p-2}, we have
\[
0>K_W'(2)=(\varphi_W(p)/p)'|_{p=2}=\frac{2 \varphi_W'(2)-\varphi_W(2)}{4}.
\]
That is $\varphi_W(2) > 2 \varphi_W'(2)$. Now by combining with  the equalities \eqref{psi-1} and \eqref{psi-p-1}, we obtain  
\[
\psi(1) =\varphi_W(2) >2 \varphi_W'(2) = \psi'(1). 
\]
This completes the proof of the desired inequalities \eqref{equi-cond}.  
\end{proof}

\subsection{CLT for rescaled $\widehat{\mu}_\infty(2^n)$}

For proving  Proposition~\ref{prop-convergence-in-law}, we are going to apply the conditional Lindeberg-Feller central limit theorem (see, e.g., \cite[Proposition A. 3]{CHQW24a} for a version that is convenient for our purpose). 
By  Lemma~\ref{lem-self-similar} and the equality $2^{-D_F}=2\E[W_0^2]$, using the definition \eqref{eqn-def-yu} of $Y(u)$, we get 
\begin{align}\label{d=sum}
2^{\frac{nD_F}{2}}\widehat{\mu}_\infty(2^n)\overset{d}{=}\frac{1}{(8\E[W_0^2])^{\frac{n}{2}}}\sum_{|u|=n}\prod_{j=1}^{n} X(u|_j)\widehat{\mu}_\infty^{(u)}(1)=\sum_{|u|=n}\sqrt{Y(u)}\widehat{\mu}_\infty^{(u)}(1).
\end{align}
Note that,     conditioned on $\mathscr{F}_n$,  the random variables in the family 
\[
 \big\{\sqrt{\VU}  \cdot \widehat{\mu}_{\infty}^{(u)}(1)\big\}_{|u|=n}.
\]
are conditionally centered and independent.

\begin{lemma}\label{lemma-hyyyy}
We have 
\begin{align}\label{Yu-to-0}
\lim_{n\rightarrow\infty}\sup_{|u|=n} Y(u)=0,\quad a.s.
\end{align}
\end{lemma}

\begin{proof}
For $|u|=n$,
recall the definition  \eqref{eqn-def-vu} of $V(u)$, one has
\[
-\log Y(u)=n\log (8\E[W^2])-2\sum_{j=1}^{n}\log X(u|_j)=V(u)+n\varphi_W(2).
\]
By \cite[Theorem 1.3]{shi-zhan-book}  and the equality \eqref{phi-psi}, 
\begin{align*}
\lim_{n\rightarrow\infty}\inf_{|u|=n} \frac{-\log Y(u)}{n}& =\lim_{n\rightarrow\infty}\inf_{|u|=n} \frac{V(u)}{n}+\varphi_W(2)
\\
& =-\inf_{\beta>0}\frac{\psi(\beta)}{\beta}+\varphi_W(2)
\\
&=  - \inf_{\beta>0} \frac{\varphi_W(2\beta)}{\beta} + \varphi_W(2).
\end{align*}
Take $\beta=3/2$, by the proof of Lemma~\ref{lem-3-2}, we get that  
\[
\varphi_W(3)/3<\varphi_W(2)/2.
\]
This implies that 
\[
\lim_{n\rightarrow\infty}\sup_{|u|=n} \frac{\log Y(u)}{n}=\inf_{\beta>0} \frac{\varphi_W(2\beta)}{\beta} - \varphi_W(2)<0.
\]
The desired convergence \eqref{Yu-to-0} follows immediately.
\end{proof}

\begin{lemma}\label{lemma-lindeberg-feller}
 For any $\varepsilon>0$,  the following almost sure convergence holds:
\begin{align}\label{sum-Yu-0}
\lim_{n\rightarrow \infty} \sum_{|u|=n} \E[Y(u) |\widehat{\mu}_{\infty}^{(u)}(1)|^2\mathds{1}(Y(u) |\widehat{\mu}_{\infty}^{(u)}(1)|^2>\varepsilon)|\mathscr{F}_n]=0.
\end{align}
\end{lemma}
\begin{proof}
Denote by 
\[
\sigma(x): =\E\big[|\widehat{\mu}_{\infty}(1)|^2 \mathds{1}(|\widehat{\mu}_{\infty}(1)|^2>x)\big], \quad \forall x \geqslant 0.
\]
Clearly, $\sigma(x)$ is non-increasing for $x\ge 0$.  Since $\E[|\widehat{\mu}_{\infty}(1)|^2]=\varrho<+\infty$, the Dominated Convergence Theorem implies 
\[
\sigma(x)\downarrow 0 \quad \text{as $x\uparrow \infty $}.
\]
Since for $|u|=n$, the random variable  $Y(u)$ is $\mathscr{F}_n$-measurable and $\mu_\infty^{(u)}(1)$ is independent of $\mathscr{F}_n$,  we have 
\begin{align*}
&  \sum_{|u|=n} \E\Big[Y(u) |\widehat{\mu}_{\infty}^{(u)}(1)|^2\mathds{1}(Y(u) |\widehat{\mu}_{\infty}^{(u)}(1)|^2>\varepsilon)\Big| \mathscr{F}_n\Big]\\
  =& \sum_{|u|=n} Y(u)\sigma\Big(\frac{\varepsilon}{Y(u)}\Big) \le \mathcal{M}_n^{(2)}\cdot\sigma\Big(\frac{\varepsilon}{\sup_{|u|=n}\VU}\Big). 
\end{align*}
Therefore, the desired almost sure convergence \eqref{sum-Yu-0} follows from Lemma~\ref{lemma-hyyyy} and  the almost sure convergence of the martingale  $(\mathcal{M}_n^{(2)})_{n\ge 1}$.  
\end{proof}
\begin{proof}[Proof of Proposition~\ref{prop-convergence-in-law}]
This follows from Lemma~\ref{lemma-lindeberg-feller} and the conditional Lindeberg-Feller central limit theorem (see, e.g., \cite[Proposition A. 3]{CHQW24a} for a version that is convenient for our purpose). 

Indeed, set 
\[
V_n = 2^{\frac{n D_F}{2}}     \widehat{\mu}_{\infty}(2^n)
\]
and 
\begin{align*}
U_n: =&  \frac{1}{\sqrt{\varrho+\varpi}} \mathrm{Re}(V_n)  +  \frac{i}{\sqrt{\varrho -\varpi}} \mathrm{Im} (V_n) 
\\
=&   2^{\frac{n D_F}{2}}    \Big[ \frac{1}{\sqrt{\varrho+\varpi}} \mathrm{Re} \big(\widehat{\mu}_{\infty}(2^n)\big)   +  \frac{i}{\sqrt{\varrho -\varpi}} \mathrm{Im} \big(\widehat{\mu}_{\infty}(2^n)\big)   \Big].
\end{align*}
It suffices to show that 
\begin{align}\label{conv-Un}
U_n \xrightarrow[n\to\infty]{d} \sqrt{\mathcal{M}_\infty^{(2)}} \cdot \mathcal{N}_\C(0, 1),
\end{align}
where $\mathcal{N}_\C(0,1)$ is the standard complex Gaussian random variable which is independent of $\mathcal{M}_\infty^{(2)}$.  

By Lemma \ref{lemma-mean-zero-s}, we have 
\begin{align*}
\E[|V_n|^2|\mathscr{F}_n] =   \E[ (\mathrm{Re}(V_n))^2 |\mathscr{F}_n] +   \E[ (\mathrm{Im}(V_n))^2 |\mathscr{F}_n] =  \varrho \cdot \mathcal{M}_n^{(2)}.
\end{align*}
And, since $\E[V_n^2|\mathscr{F}_n] \in \R$,  we have 
\begin{align*}
\E[V_n^2|\mathscr{F}_n] = \mathrm{Re} (\E[V_n^2|\mathscr{F}_n] )  =   \E[ (\mathrm{Re}(V_n))^2 |\mathscr{F}_n]  -    \E[ (\mathrm{Im}(V_n))^2 |\mathscr{F}_n] =  \varpi \cdot \mathcal{M}_n^{(2)}
\end{align*}
and 
\[
0 = \mathrm{Im} (\E[V_n^2|\mathscr{F}_n])=   2 \E[\mathrm{Re} (V_n)\cdot \mathrm{Im}(V_n)|\mathscr{F}_n] . 
\]
Thus 
\[
 \E[ (\mathrm{Re}(V_n))^2 |\mathscr{F}_n]  = \frac{\varrho+\varpi}{2} \mathcal{M}_\infty^{(2)} \an \E[ (\mathrm{Im}(V_n))^2 |\mathscr{F}_n]  = \frac{\varrho-\varpi}{2} \mathcal{M}_n^{(2)}.
\]
It follows that 
\begin{align*}
\E[|U_n|^2|\mathscr{F}_n]  =  \frac{ \E[ |\mathrm{Re}(V_n)|^2 |\mathscr{F}_n]}{\varrho+\varpi} +  \frac{ \E[ |\mathrm{Im}(V_n)|^2 |\mathscr{F}_n]}{\varrho-\varpi} = \mathcal{M}_n^{(2)}
\end{align*}
and 
\[
\E[U_n^2|\mathscr{F}_n]  = 0.
\]

By \eqref{d=sum}, we have 
\[
U_n \overset{d}{=}  \sum_{|u|=n}\sqrt{Y(u)} \Big[ \frac{1}{\sqrt{\varrho+\varpi}} \mathrm{Re}\big( \widehat{\mu}_\infty^{(u)}(1)\big) + \frac{i}{\sqrt{\varrho - \varpi}}\mathrm{Im}\big( \widehat{\mu}_\infty^{(u)}(1)\big) \Big] .
\]
Then, by Lemma \ref{lemma-lindeberg-feller}, we conclude that 
the random variables $U_n$ satisfy all the assumptions of   the conditional Lindeberg-Feller central limit theorem stated in  \cite[Proposition A. 3]{CHQW24a}. Therefore,   we  complete the proof of the desired convergence in law \eqref{conv-Un}. 
\end{proof}

\section{Proof of Theorem \ref{main-thm}}\label{sec-main-thm}

Recall the following elementary lemma in   \cite[Lemma 9.4]{CHQW24a}.
\begin{lemma}\label{lemma-cov-law-con-one}
Suppose that a sequence of complex random variables $(Z_n)_{n\geq 1}$ satisfies  that $Z_n\xrightarrow[n\rightarrow\infty ]{d} Z$, 
where the random variable $Z_\infty\neq 0$ almost surely. Then for any positive increasing sequence $(a_n)_{n\in\N}$ tending to $\infty$,  one has
\[
\lim_{n\rightarrow\infty}a_n|Z_n|= \infty,\quad \text{in probability}.
\]
That is, for any $C>0$, 
\[
\lim_{n\rightarrow\infty} \PP(a_n|Z_n|>C)=1.
\]
\end{lemma}

\begin{proof}[Proof of Theorem~\ref{main-thm}]
By Proposition~\ref{prop-vec}, for any $\varepsilon>0$,  there exists $q>2$ large enough such that  
\[
 \Big\{\sum_{n\in \Z} \big(|n|^{\frac{D_F}{2}-\varepsilon}   \cdot |\widehat{\mu}_\infty(n)|\big)^q\Big\}^{2/q}<\infty, \quad  a.s. 
\]
It follows that 
\[
|\widehat{\mu}_\infty(n)|^2 = O(|n|^{-D_F+2\varepsilon}) \quad a.s. 
\]
By \cite[Lemma 1.8 or Remark 1.2]{CHQW24a},  almost surely, one has 
\[
\dim_F(\mu_\infty)\ge D_F- 2 \varepsilon.
\] Since $\varepsilon>0$ is arbitrary, we obtain the desired almost sure lower estimate  $\dim_F(\mu_\infty)\ge D_F$. 

Conversely,  by Proposition~\ref{prop-convergence-in-law} and Lemma~\ref{lemma-minfinity-positive}, one has 
\[
  2^{\frac{n D_F}{2}}  \cdot \widehat{\mu}_\infty(2^{n})  \xrightarrow[n\to\infty]{d}   \mathcal{Y}_\infty \quad   \text{with  $\PP ( \mathcal{Y}_\infty \ne 0)   = 1$.}
\]
Then,   for any $\varepsilon>0$,     choosing $a(n) =2^{n\varepsilon}$ in  Lemma \ref{lemma-cov-law-con-one}, we have 
\[
\lim_{n\rightarrow\infty}2^{\frac{n (D_F+\varepsilon)}{2}} |\widehat{\mu}_\infty(2^{n})| =\infty \quad \text{in probability}. 
\]
Therefore, there exists a   subsequence $(n_k)$ such that 
\[
\lim_{k\rightarrow\infty}2^{\frac{n_k (D_F+\varepsilon)}{2}} |\widehat{\mu}_\infty(2^{n_k})| =\infty \quad \text{a.s.}
\]
This implies the desired almost  sure upper estimate
$
\dim_F(\mu_\infty)\leq D_F$.
\end{proof}

\section{H\"older Continuity}\label{sec-Holder}

\subsection{The ranges of $\gamma_o^{+}$ and $\gamma_o^{-}$}
Recall the definitions of $\gamma_{o}^{+}$ in \eqref{eqn-opt} and $\gamma_{o}^{-}$ in \eqref{eqn-OPT}:
\begin{align*}
 \gamma_{o}^{+} &=  \gamma_{o}^{+}(W) :=\sup_{p>0} \frac{\log_2\big [ (\E[W_0^p]+\E[W_1^p])^{-1}\big] }{p} ;
 \\
 \gamma_{o}^{-}  &=   \gamma_{o}^{-} (W): =  \inf_{p>0} \frac{\log_2 \big[ \E[W_0^{-p}]+\E[W_1^{-p}]\big] }{p} . 
 \end{align*}
 Recall also that, by assumption, $W_0$ is not identically $1/2$ and $0<W_0<1$ a.s.  (hence $0<W_1<1$ a.s. since $W_1 = 1-W_0$). 
\begin{lemma}\label{lem-opt}
We have $0<\gamma_{o}^{+} <1$.
\end{lemma}

\begin{proof}
Note that for any $t, s$ with $0<t\le 1\le s<\infty$, we have 
\[
 (W_0^t + W_1^t)^{1/t} \ge  W_0 + W_1 =  1 \ge (W_0^s + W_1^s)^{1/s}
\]
and hence 
\[
\gamma_o^{+} =  - \inf_{p >0}   \log_2  \big[(\E[W_0^p + W_1^p])^{1/p}\big] =   - \inf_{p \ge 1}   \log_2  \big[(\E[W_0^p + W_1^p])^{1/p}\big].  
\]
Clearly,  $\gamma_o^{+}>0$, since $[(\E[W_0^p + W_1^p])^{1/p}]<1$ for all $p>1$.  

It remains to prove that $\gamma_o^{+}<1$. Indeed,  for any $p\ge 1$, we have 
\[
\log_2 \big[(\E[W_0^p + W_1^p])^{1/p}\big] \ge \log_2\big[ (\E[(\max\{W_0,W_1\})^p] )^{1/p} \big]  \ge \log_2 \E[\max\{W_0,W_1\}].
\]
Note that   $W_1=1-W_0$ and $W_0 \not\equiv1/2$,  hence  $\max\{W_0, W_1\} > 1/2$.  It follows that 
$\gamma_o^{+}< 1$. 
\end{proof}

\begin{lemma}\label{lem-Opt} 
We have  $\gamma_{o}^{-}\in (1,\infty]$.
\end{lemma}
\begin{proof}
If $\E[W_0^{-p} + W_1^{-p}]= \infty$ for any $p>0$, then we have $\gamma_o^{-} = \infty$. 

Now assume that there exists $p_0>0$ such that $\E[W_0^{-p_0} +W_1^{-p_0}]<\infty$.  Then $\gamma_o^{-}<\infty$. We shall prove that in this case,  $\gamma_o^{-} >1$.   Indeed, under the assumption $\E[W_0^{-p_0} +W_1^{-p_0}]<\infty$, we have 
\[
\lim_{p\to 0^{+}} \frac{ \log_2 \E[W_0^{-p}  + W_1^{-p}]}{p}  = +  \infty. 
\]
Hence there exists $p_1$ with  $0<p_1< p_0$ such that  
\[
\gamma_o^{-} =  \inf_{p>0} \frac{\log_2 \big[ \E[W_0^{-p}]+\E[W_1^{-p}]\big] }{p}  =  \inf_{p\ge p_1} \frac{\log_2 \big[ \E[W_0^{-p}]+\E[W_1^{-p}]\big] }{p}. 
\]
Note that, for any $p\ge p_1$,  
\begin{align*}
( \E[W_0^{-p} + W_1^{-p}])^{1/p}  & \ge  (\E[(\max\{W_0^{-1},W_1^{-1}\})^p] )^{1/p}  
\\
&=  (\E[(\min\{W_0,W_1\})^{-p}] )^{1/p} 
\\
& \ge  (\E[(\min\{W_0,W_1\})^{-p_1}] )^{1/p_1}.  
\end{align*}
Hence 
\[
\gamma_o^{-} =  \inf_{p\ge p_1}  \log_2 \big[ \big(\E[W_0^{-p}]+\E[W_1^{-p}]\big)^{1/p}\big]  \ge  \log_2\big[ (\E[(\min\{W_0,W_1\})^{-p_1}] )^{1/p_1}\big] . 
\]
Finally,  since $W_1=1-W_0$ and $W_0 \not\equiv 1/2$, we have $\min\{W_0, W_1\} < 1/2$ and hence 
\[
(\E[(\min\{W_0,W_1\})^{-p_1}] )^{1/p_1}>2. 
\]
It follows that $\gamma_o^{-}>1$. 
\end{proof}

\subsection{Proof of Proposition \ref{frostman-measure}}
Let $\mathscr{D}$ denote the collection of all dyadic subintervals of $[0,1)$. 
By a routine  standard argument, to prove the inequalities  \eqref{ineq-2-frost},  it suffices to prove that, 
\begin{align}\label{check-dya}
\frac{1}{C}  |I|^{\gamma_o^{-}} \le \mu_\infty(I) \le C |I|^{\gamma_o^{+}} \quad \text{for all $I\in \mathscr{D}$}. 
 \end{align}
 We will give the proof of the right-hand side  of \eqref{check-dya}. The left-hand side can be handled using the same method.
Take $v$ with $|v|=k$.  Then by definition \eqref{def-mun},  for any $n\ge k +1$, 
\begin{align*}
\mu_n(I_v)&=  \sum_{|u|=n} \Big( \prod_{j=1}^n X(u|_j)\Big) \cdot |I_u\cap I_v|
\\
& =  \sum_{|w|=n-k}   \Big( \prod_{j=1}^k X(v|_j) \Big) \cdot \Big( \prod_{l=1}^{n-k} X(v \cdot w|_l)\Big)   \cdot \frac{1}{2^n} 
\\
 &=   \Big(\prod_{j=1}^k  \frac{ X(v|_j)}{2}\Big) \cdot  \sum_{|w| = n-k}   \prod_{l=1}^{n-k}\frac{X(v \cdot w|_l)}{2}. 
\end{align*}
Since for any $m\ge 1$, 
\begin{align*}
\sum_{|w| = m}  \prod_{l=1}^{m} \frac{X(v \cdot w|_l)}{2} &  =  \sum_{w_1, \cdots, w_m \in \{0,1\}}   W_{w_1}(v)  W_{w_2}(vw_1)\cdot  \cdots    \cdot W_{w_m}(v w_1\cdots w_{m-1})
\\
& = \sum_{w_1, \cdots, w_{m-1}\in \{0,1\}}   W_{w_1}(v)  W_{w_2}(vw_1)\cdot  \cdots    \cdot W_{w_{m-1}}(v w_1\cdots w_{m-2})
\\
& = \cdots = 1,
\end{align*}
where we used the fact that $W_0(u) + W_1(u) = 1$ for any $u$.  
It follows that 
\begin{align}\label{eqn-muinfinityiv}
\mu_\infty(I_v) = \lim_{n\to\infty} \mu_n(I_v)= \prod_{j=1}^k  \frac{ X(v|_j)}{2} =  \prod_{j=1}^k W_{v_j}(v|_{j-1}). 
\end{align}

Let
\[
\xi_0(u) := -\log W_0(u) \an   \xi_1(u) := -\log W_1(u).
\]
For any $v$ with $|v|=k\ge1$, set
\[
S_v := \sum_{j=1}^k \xi_{v_j}(v|_{j-1}) \an S_\varnothing =0.
\]
 Then $(S_u)_{u\in\mathcal{T}_2}$ forms a branching random walk with reproduction law given by 
\begin{align}\label{eqn-bianhuan}
(\xi_0, \xi_1)= (-\log W_0, -\log W_1).
\end{align}
By \eqref{eqn-muinfinityiv},  for any $v$ with $|v|=k$,
\[
\mu_\infty(I_v) =  \prod_{j=1}^k W_{v_j}(v|_{j-1}) = \exp(-S_v).
\]
Thus,
\begin{equation}\label{maxweight}
\sup_{|v|= k} \mu_\infty(I_v) = \exp\Big( - \inf_{|v|=k}S_v\Big)  \an  \inf_{|v|=k}\mu_\infty(I_v) = \exp\Big(- \sup_{|v|=k}S_v\Big).
\end{equation}
Recall the function $\varphi_W(p)$  defined in \eqref{eqn-varphiwp}. By \eqref{eqn-bianhuan}, we  have
\[
\varphi _W(p) = \log \E[e^{-p\xi_0} + e^{- p \xi_1}] \in (-\infty, +\infty].
\]
Observe by \eqref{def-W0W1}, $\varphi_W(1 ) =0$. Therefore, by   \cite[Theorem 1.3]{shi-zhan-book}, 
\[
\frac1n \inf_{|u|=n}S_u \xrightarrow[a.s.]{n\to\infty} \gamma_o^+ \log  2.
\]
It is known from    \cite[Theorem 3]{Biggins98} that if there exists some $p_0>0$ such that
\[
\gamma_o^+ \log 2=  \frac{\log \big [ (\E[W_0^{p_0}]+\E[W_1^{p_0}])^{-1}\big] }{p_0} = - \inf_{t>0} \frac{\varphi_W(t) }{t} \in\R,
\]
then 
\[
\inf_{|u|=n}S_u - n\gamma_o^+\log 2 \xrightarrow[a.s.]{n\to\infty} +\infty.
\]
Going back to \eqref{maxweight}, we get that
\[
\sup_{|v|= k} \frac{ \mu_\infty(I_v) }{|I_v|^{\gamma_o^+}} = 2^{k\gamma_o^+} \sup_{|v|= k} \mu_\infty(I_v) =   \exp\Big(- \big(\inf_{|v|=k}S_v - k\gamma_o^+\log 2\big) \Big) \xrightarrow[a.s.]{k\to\infty} 0.
\]
Therefore, we have a.s.,
\[
\sup_{k\ge 1} \sup_{|v|= k} \frac{ \mu_\infty(I_v) }{|I_v|^{\gamma_o^+}} <\infty .
\]
This is sufficient to conclude and completes the whole proof. 

\end{document}